\documentclass[a4paper,12pt,leqno]{amsart}
\usepackage[english]{babel}
\usepackage{amsmath}
\usepackage{amscd}
\usepackage{amsgen}
\usepackage[final]{epsfig}
\usepackage{latexsym}
\usepackage{amsfonts}
\usepackage{amssymb}
\usepackage{amsthm}
\usepackage{url}
\usepackage{slashed}

\allowdisplaybreaks

\theoremstyle{plain}

\newtheorem{thm}{Theorem}[section]
\newtheorem{lemm}{Lemma}[section]
\newtheorem{conj}{Conjecture}[section]
\newtheorem{corr}{Corollary}[section]
\newtheorem{defn}{Definition}[section]

\newtheorem{ex}{Example}[section]

\newcommand{\Hess}{\operatorname{Hess}}
\newcommand{\Ric}{\operatorname{Ric}}
\newcommand{\tr}{\operatorname{tr}}
\newcommand{\Res}{\operatorname{Res}}

\def\r{{\mathbb R}}        %R
\def\n{{\mathbb N}}        %N
\def\h{{\mathbb H}}        %H
\def\S{{\mathbb S}}        %S
\def\st{\stackrel{\text{def}}{=}}

\def\B{{\mathcal B}}     %Bach tensor
\def\C{{\sf C}}          %Weyl tensor
\def\CC{{\bf C}}
\def\L{{\mathcal L}}
\def\G{{\mathcal G}}
\def\M{{\mathcal M}}

\def\OB{{\mathcal O}}
\def\P{{\mathcal P}}
\def\Q{{\mathcal Q}}
\def\R{{\mathcal R}}
\def\T{{\mathcal T}}
\def\V{{\mathcal V}}
\def\Z{{\mathbb Z}}
\def\Y{{\mathcal C}}     %Cotton tensor

\def\Rho{{\sf P}}        %Rho-Tensor
\def\PP{{\bf P}}

\def\J{{\sf J}}
\def\f{{\frac{n}{2}}}

\numberwithin{equation}{section}

\title{On conformally covariant powers of the Laplacian}

\author{Andreas Juhl}
\address{Humboldt-Universit\"at, Institut f\"ur Mathematik,
Unter den Linden, 10099 Berlin, G
ermany}

\email{ajuhl@math.hu-berlin.de}

\begin{document}

\begin{abstract} We propose and discuss recursive formulas for
conformally covariant powers $P_{2N}$ of the Laplacian
(GJMS-operators). For locally conformally flat metrics, these
describe the non-constant part of any GJMS-operator as the sum of a
certain linear combination of compositions of lower order
GJMS-operators (primary part) and a second-order operator which is
defined by the Schouten tensor (secondary part). We complete the
description of GJMS-operators by proposing and discussing recursive
formulas for their constant terms, i.e., for Branson's
$Q$-curvatures, along similar lines. We confirm the picture in a
number of cases. Full proofs are given for spheres of any dimension
and arbitrary signature. Moreover, we prove formulas of the
respective critical third power $P_6$ in terms of the Yamabe
operator $P_2$ and the Paneitz operator $P_4$, and of a fourth power
in terms of $P_2$, $P_4$ and $P_6$. For general metrics, the latter
involves the first two of Graham's extended obstruction tensors
\cite{G-ext}. In full generality, the recursive formulas remain
conjectural. We describe their relation to the theory of residue
families and the associated $Q$-polynomials as developed in
\cite{J-book}.
\end{abstract}

\maketitle

\centerline \today

\renewcommand{\thefootnote}{}

\footnotetext{The work was supported by SFB 647 "Raum-Zeit-Materie"
of DFG.}


\tableofcontents

%%%%%%%%%%%%%%%%%%%%%%%%%%%%%%%%%%%%%%%%%%%%%%%%%%%%%%%%%%%%%%%%%%%%%%%%%%%%%%%%
\section{Introduction}

The Laplace-Beltrami operator $\Delta_g$ of a Riemannian manifold
$(M,g)$ is one of the basic geometric differential operators. Its
significance rests on its invariance with respect to isometries. In
two dimensions, it is also invariant (or rather covariant) with
respect to conformal changes $g \mapsto e^{2\varphi} g$ of the
metric. Although this is not true in dimension $n \ge 3$, the
operator
\begin{equation}\label{yama}
P_2(g) = \Delta_g - \left(\f-1\right) \J_g, \quad \J_g =
\tau_g/2(n-1),
\end{equation}
which arises by addition of a multiple of the scalar curvature
$\tau$, is conformally covariant, i.e.,
$$
e^{(\f+1)\varphi} \circ P_2(e^{2\varphi}g)  = P_2(g) \circ
e^{(\f-1)\varphi}
$$
for all $\varphi \in C^\infty(M)$ (here the functions $e^{\varphi}$
act as multiplication operators). The operator \eqref{yama} is known
as the conformal Laplacian or Yamabe operator. It plays a central
role in conformal geometry and related geometric analysis. Here and
throughout, we use the convention that $-\Delta$ is non-negative.

About twenty five year ago, a conformally covariant operator of the
form $\Delta^2 + LOT$ was discovered independently by Paneitz
\cite{P}, Eastwood-Singer \cite{ES} and Riegert \cite{Rieg}; $``LOT"$
indicates terms with fewer than four derivatives. On manifolds of
dimension $n \ge 3$, it is defined by
\begin{equation}\label{pan}
P_4 = \Delta^2 + \delta ((n-2) \J g - 4 \Rho) \# d +
\left(\f-2\right) \left(\f \J^2 - 2 |\Rho|^2 - \Delta \J \right),
\end{equation}
where $\Rho$ is the Schouten tensor, i.e., $(n-2) \Rho = \Ric - \J
g$, $\#$ denotes the natural action of symmetric bilinear forms on
$1$-forms and $\delta$ is the formal adjoint of the differential $d$.
$P_4$ satisfies the transformation law
$$
e^{(\f+2)\varphi} \circ P_4(e^{2\varphi}g) = P_4(g) \circ
e^{(\f-2)\varphi}, \; \varphi \in C^\infty(M).
$$
A significant difference between \eqref{yama} and \eqref{pan} is the
appearance of the Ricci tensor in the Paneitz operator. The scalar
curvature quantity
\begin{equation}\label{q4-gen}
Q_4 = \f \J^2 - 2 |\Rho|^2 - \Delta \J
\end{equation}
in the constant term of $P_4$ is a special case of Branson's
$Q$-curvature \cite{bran-2}. For $n=4$, the fourth-order curvature
quantity $Q_4$ satisfies the remarkable transformation law \cite{bo}
\begin{equation}\label{Q4}
e^{4\varphi} Q_4(e^{2\varphi}g) = Q_4(g) + P_4(g)(\varphi),
\end{equation}
which generalizes
\begin{equation}\label{gauss}
e^{2\varphi} Q_2(e^{2\varphi}g) = Q_2(g) - P_2(g)(\varphi)
\end{equation}
for $Q_2 = \J = \tau/2$ in dimension $2$. Since $\tau/2$ is the Gau{\ss}
curvature and $P_2 = \Delta$, \eqref{gauss} is nothing else than the
Gau{\ss} curvature prescription equation. The $Q$-curvature prescription
equation \eqref{Q4} has been at the center of much research in
recent years (see \cite{malch} for a review).

The discovery of $P_4$ naturally raised the problem of constructing
higher order analogs, i.e., of similarly correcting any power
$\Delta^N$ of the Laplacian by appropriate lower order terms so that
the resulting operator becomes conformally covariant. For $N=3$,
such results are already contained in \cite{bran-1}. The
construction in \cite{GJMS} of conformally covariant powers of the
Laplacian in terms of the powers of the Laplacian for the
Fefferman-Graham ambient metric \cite{cartan}, \cite{FG2} settled
the existence problem. In addition, it revealed obstructions to
their existence on even dimensional manifolds. In the following, we
shall refer to the operators constructed in \cite{GJMS} as the
GJMS-operators, and denote them by $P_{2N}$. On a manifold of even
dimension $n$, the GJMS-operator of order $n$ will be called the
{\em critical} GJMS-operator. For more details we refer to Section
\ref{m-def}.

The Yamabe operator \eqref{yama} and the Paneitz operator
\eqref{pan} are the first two GJMS-operators. For higher orders, the
{\em structure} of the GJMS-operators remained obscure up to now,
and it is generally believed that explicit formulas for them are
hopelessly complicated due to the exponential increase of their
complexity as a function of the order. It is tempting to compare
this with the complexity of heat kernels.

One of the remarkable properties of the GJMS-operators is that,
through conformal variation, they are determined by their constant
terms. More precisely,
\begin{equation}\label{sub-det}
(d/dt)|_0 \left(e^{2t\varphi} Q_{2N} (e^{2t\varphi} g) \right) =
(-1)^N P_{2N}^0(g)(\varphi),
\end{equation}
where
$$
P_{2N}(1) = (-1)^N \left(\f-N\right) Q_{2N},
$$
and $P_{2N}^0$ denotes the non-constant part of $P_{2N}$. We shall
refer to the quantities $Q_{2N}$ as Branson's $Q$-curvatures (see
Section \ref{m-def}) although sometimes only the critical
$Q$-curvature $Q_n$ bears that name. Thus, an understanding of the
GJMS-operators is intimately connected with an understanding of the
structure of the $Q$-curvatures. Since the complexity of
$Q$-curvatures exponentially increases as well, it is generally
believed that aiming for explicit formulas for high order
$Q$-curvatures is also hopeless. Even in the presence of
well-structured formulas for $Q$-curvatures, it remains a
non-trivial problem to derive such formulas for the corresponding
GJMS-operators by conformal variation.

On the other hand, motivated by the rich results in geometric
analysis around $P_4$ and $Q_4$, explicit formulas for high order
$Q$-curvatures and GJMS-operators are of substantial interest.
Uncovering their structure could open the way to future geometric
applications. Presently, this is an almost unexplored area.

%revised

A remarkable exception is the work \cite{GP}. It addresses the
problem to find explicit formulas for GJMS-operators from the point
of view of tractor calculus. Gover and Peterson describe an
algorithm for deriving explicit formulas for these operators in
terms of tractor constructions. An evaluation of the algorithm for
$P_8$ in terms of the Levi-Civita connection and its curvature
generates an explicit formula which already occupies several pages.
In the opinion of the present author, these results supported the
belief that the structures of the operators $P_{2N}$ (and the
related $Q$-curvatures $Q_{2N}$) are hopelessly complicated.

The moral of the present paper is that, in contrast to the accepted
opinion, the complexity of high order GJMS-operators and
$Q$-curvatures is strongly tamed by beautiful recursive structures.
More precisely, we formulate systems of conjectural recursive
relations among GJMS-operators and $Q$-curvatures, and describe how
these would lead to explicit formulas. The relations are summarized
in Conjecture \ref{formula}, Conjecture \ref{Q-quadrat} and
Conjecture \ref{Q-G}.

%revised

Conjecture \ref{formula} is supported by complete proofs of the
corresponding formulas for $P_6$ (in general dimensions and for
general metrics) and $P_8$ (in the critical dimension and for
general metrics) as well as by basic structural results along the
lines of these conjectures. We prove that the GJMS-operators on the
conformally flat round spheres $\S^n$ are captured by the recursive
algorithm and confirm an extension of that result to the conformally
flat pseudo-spheres $\S^{q,p}$. Similarly, Conjecture
\ref{Q-quadrat} and Conjecture \ref{Q-G} are supported by proofs for
$Q_{2N}$ with $N \le 4$ for general metrics and proofs for all
spheres $\S^n$ and pseudo-spheres $\S^{q,p}$.

Although a complete understanding of the picture requires much more
efforts, it is tantalizing to regard its overall simplicity as an
argument in its favor.

The main features of the proposed recursive formulas for
GJMS-operators are the following.
\begin{itemize}
\item Any GJMS-operator is described by a primary part, a secondary part
(given by a second-order operator), and a constant term (given by
$Q$-curvature).
\item The primary parts are defined in terms of universal linear
combinations of compositions of respective lower order
GJMS-operators.
\end{itemize}
Here universality means that the coefficients of the linear
combinations do not depend on the dimension of the underlying
manifold.

In order to complete the description of GJMS-operators, we propose a
recursive description of $Q$-curvatures along similar lines. The
main features of the proposed recursive formulas for $Q$-curvatures
are the following.
\begin{itemize}
\item Any $Q$-curvature is described as the sum of a primary
and a secondary part.
\item The primary parts of GJMS-operators and $Q$-curvatures are
linked to each other.
\item The secondary parts of $Q$-curvatures are given by universal
formulas in terms of holographic coefficients.
\end{itemize}

The holographic coefficients (or renormalized volume coefficients
\cite{G-vol}, \cite{G-ext}) are functionals of a metric which arise
as the coefficients in the Taylor expansion of the volume form of an
associated Poincar\'e-Einstein metric. They are locally determined
by the metric and can be written in terms of (derivatives of) the
curvature tensor. One of these quantities plays the role of a
conformal anomaly of the renormalized volume. The latter concept was
introduced in connection with the AdS/CFT duality \cite{HS},
\cite{W}, \cite{G-vol}, \cite{and}. For more information see Section
\ref{Q-rec} and \cite{J-book}.

In connection with the recursive formulas for $Q$-curvatures, the
principle of universality means that formulas in the critical
dimension literally hold true also in the subcritical cases.

We illustrate the recursive structure of GJMS-operators by means of
the conformally covariant third power $P_6$ of the Laplacian. For
this purpose, we restrict to locally conformally flat metrics and
comment only briefly on the general case. First of all, on locally
conformally flat manifolds of dimension $n=6$, the self-adjoint
operator
\begin{equation}\label{illu}
\PP_6 \st \big[2(P_2 P_4 + P_4 P_2) - 3 P_2^3\big]^0 - 48 \delta
(\Rho^2 \# d) = \Delta^3 + LOT
\end{equation}
is conformally covariant, i.e., satisfies
$$
e^{6\varphi} \PP_6(e^{2\varphi}g) = \PP_6(g)
$$
for all $\varphi \in C^\infty(M)$. Here $[\cdot]^0$ denotes the
non-constant part of the respective operator in brackets. Moreover,
the operator $\PP_6$ coincides with the critical GJMS-operator
$P_6$. These results were first obtained in \cite{J-book}. An
alternative proof of the conformal covariance of a generalization of
$\PP_6$ for general metrics will be given Section \ref{inv-4-6}.

Of course, the formula \eqref{illu} is not explicit in terms of the
Levi-Civita connection and its curvature. But such formulas easily
follow from \eqref{illu} by using the formulas \eqref{yama} and
\eqref{pan} for $P_2$ and $P_4$. Although the resulting expressions
might be interesting in connection with applications to geometric
analysis, they will hide the recursive structure expressed by
\eqref{illu}.

Now the right-hand side of \eqref{illu} is the sum of the
non-constant part of the {\em primary part}
\begin{equation}\label{prime}
\P_6 = 2P_2P_4 + 2P_4P_2 - 3P_2^3
\end{equation}
and the {\em secondary part}, which is a multiple of the
second-order operator $\delta(\Rho^2 \#d)$; the constant term
vanishes in this case. For general metrics, the secondary part
contains an additional second-order operator which is defined by the
Bach tensor (or rather Graham's \cite{G-ext} first extended
obstruction tensor).

In the locally conformally flat category, Theorem \ref{M8} yields a
similar formula
$$
\PP_8 \st \P_8^0 - c_4 \delta(\Rho^3 \# d), \; c_4 = 3!4!2^3
$$
for a conformally covariant operator of the form $\Delta^4 + LOT$ in
dimension $n=8$. Here the primary part $\P_8$ is a certain linear
combination of all possible compositions of lower order
GJMS-operators to an operator of order $8$. The secondary part is a
second-order operator which is determined by the Schouten tensor
$\Rho$. For general metrics, Theorem \ref{P8-non-flat} shows that
the corresponding secondary parts involve additional contributions
of the first two extended obstruction tensors $\Omega^{(1)}$ and
$\Omega^{(2)}$.

The last point is worth emphasizing. The operators $P_2$, $P_4$ and
$P_6$ are generated by linear combinations of compositions of
respective lower order relatives and addition of suitable
second-order correction terms. In particular, these constructions do
not involve the Fefferman-Graham ambient metric. But in $\PP_8$ the
ambient metric is {\em forced} to appear in terms of the extended
obstruction tensors.

It remains open whether $\PP_8$ coincides with $P_8$.

Conjecture \ref{formula} specifies in which sense these results are
special cases of a representation formula for {\em all}
GJMS-operator. This conjecture concerns locally conformally flat
metrics. It states that the non-constant part of $P_{2N}$ always can
be written in the form
\begin{equation}\label{fund}
P_{2N}^0 = \P_{2N}^0 - c_N \delta(\Rho^{N-1}\#d),
\end{equation}
where the primary part $\P_{2N}$ is a remarkable linear combination
of compositions of GJMS-operators of order $\le 2N-2$. For general
metrics, only the secondary part is expected to become more
complicated (generalizing Theorem \ref{P8-non-flat}).

The primary parts $\P_{2N}$ are defined in Section \ref{m-def}.
Combining \eqref{fund} with the formula \eqref{q-curv} for the
constant term of $P_{2N}$, yields a recursive formula for $P_{2N}$
in terms of lower order GJMS-operators and $Q$-curvatures. In order
to recognize the right-hand side of \eqref{fund} as the non-constant
part of a conformally covariant operator, it is a key step to prove
that the conformal variation of $\P_{2N}$ is a {\em second-order}
operator. This is done in Theorem \ref{c-cv}. Theorem \ref{comm-2}
provides an analogous treatment of the second-order secondary part
in \eqref{fund}.

In the special case of $\PP_6$ in dimension $n=6$, the variational
formula reads
\begin{equation}\label{ic-var-6}
(d/dt)|_0 \left(e^{6t\varphi} \P_6(e^{2t\varphi}g) \right) = 4
[\M_4,[P_2,\varphi]] + 2 [P_2,[\M_4,\varphi]],
\end{equation}
where $\M_4 = P_4 - \P_4$ with $\P_4 = P_2^2$. Since $\M_4$ is a
second-order operator, it follows that the right-hand side of
\eqref{ic-var-6} is a second-order operator. Combining
\eqref{ic-var-6} with the conformal variation law of $\delta(\Rho^2
\# d)$ (Theorem \ref{comm-2}), proves the conformal covariance of
$\PP_6$.

It is natural to ask why $\PP_6$ (defined by \eqref{illu}) coincides
with the GJMS-operator $P_6$. This coincidence does {\em not} follow
from the above discussion. Instead, in the locally conformally flat
case, it rests on the recursive formula
\begin{equation}\label{q6-basic}
Q_6 = \left[ -2P_2(Q_4) + 2 P_4(Q_2) - 3 P_2^2(Q_2)\right] - 6(Q_4 +
P_2(Q_2)) \cdot Q_2 + 48 \tr (\wedge^3 \Rho),
\end{equation}
which expresses the order six curvature quantity $Q_6$ in terms of
$Q$-curvatures and GJMS-operators of orders $\le 4$ (and the
Schouten tensor $\Rho$). Using conformal variation, i.e.,
$$
P_6(g)(\varphi) = -(d/dt)|_0 (e^{6t\varphi} Q_6(e^{2t\varphi}g)),
$$
it follows that $P_6$ is given by \eqref{illu} (for a detailed proof
we refer to \cite{J-book}, Section 6.12).

We illustrate the principles of the recursive description of
$Q$-curvatures by means of \eqref{q6-basic}. First of all, the sum
$$
\Q_6 \st -2P_2(Q_4) + 2 P_4(Q_2) - 3 P_2^2(Q_2)
$$
will be called the {\em primary part} of $Q_6$. The relation between
the primary part $\Q_6$ of $Q_6$ and the primary part $\P_6$ (see
\eqref{prime}) of $P_6$ is obvious: in order to get the primary part
of $Q_6$, one replaces the most right factor of each summand in the
primary part of $P_6$ by the corresponding $Q$-curvature (up to a
sign). The general case is defined in Definition \ref{Q-leading}.

\eqref{q6-basic} is a special case of Conjecture \ref{Q-quadrat}. It
relates the secondary parts of $Q$-curvatures. In fact, the
quantities $Q_4 + P_2(Q_2)$ and $Q_2$, which in \eqref{q6-basic}
contribute to the secondary part of $Q_6$, are natural relatives of
$Q_6-\Q_6$. They appear on the left-hand sides of the analogous
formulas
\begin{equation}\label{univ-q4}
Q_4 + P_2(Q_2) = - Q_2 \cdot Q_2  + 4 \tr(\wedge^2 \Rho)
\end{equation}
and
\begin{equation}\label{univ-q2}
Q_2 = \tr (\Rho).
\end{equation}
with the respective primary parts $\Q_4 = -P_2(Q_2)$ and $\Q_2 = 0$
of $Q_4$ and $Q_2$. In the other direction, the difference
$$
Q_6 - \Q_6 = Q_6 - \left[-2 P_2 (Q_4) + 2 P_4 (Q_2) - 3 P_2^2 (Q_2)
\right]
$$
contributes to the secondary part of $Q_8$.

There is an important equivalent formulation of Conjecture
\ref{Q-quadrat}. It arises as follows. Using \eqref{univ-q4} and
\eqref{univ-q2}, the presentation \eqref{q6-basic} of $Q_6$ can be
written in the alternative form
\begin{equation}\label{Q6-v}
Q_6 - \Q_6 = -48 (8 v_6 - 4 v_4 v_2 + v_2^3),
\end{equation}
where the holographic coefficients $v_{2j}$ (see \eqref{hol-coeff})
are given by
$$
v_{2j} = (-2^{-1})^j \tr (\wedge^j \Rho).
$$
\eqref{Q6-v} is an analog of
\begin{equation}\label{Q4-v}
Q_4 - \Q_4 = 4 (4 v_4 - v_2^2).
\end{equation}

The right-hand sides of \eqref{Q4-v} and \eqref{Q6-v} have the
following interpretation. Let the functions $w_{2j} \in C^\infty(M)$
be defined by the formal power series expansion
$$
\sqrt{v(r)} = 1 + w_2 r^2 + w_4 r^4 + w_6 r^6 + \cdots,
$$
where $v(r)$ is defined in \eqref{hol-coeff}. Then
\begin{equation}\label{defect-4-6}
Q_4 - \Q_4 = 2! 2^3 w_4 \quad \mbox{and} \quad Q_6 - \Q_6 = - 2! 3!
2^5 w_6.
\end{equation}
The identities in \eqref{defect-4-6} hold true in all dimensions $n
\ge 3$ (in the locally conformally flat case). These are examples of
universality.

It is convenient and natural to summarize the descriptions of the
secondary parts $Q_{2N}-\Q_{2N}$ in form of the equality (Conjecture
\ref{Q-G})
\begin{equation}\label{second-gen}
\G\left(\frac{r^2}{4}\right) = \sqrt{v(r)}
\end{equation}
of generating functions. Here
\begin{equation}\label{gf}
\G(r) \st  1 + \sum_{N \ge 1} (-1)^N (Q_{2N} - \Q_{2N})
\frac{r^N}{N!(N-1)!}.
\end{equation}
In particular, \eqref{second-gen} contains the next identity
$$
Q_8 - \Q_8 = 3! 4! 2^8 w_8
$$
(see \eqref{next} and Example \ref{qres-8}). Of course, for even $n$
and general metrics, \eqref{second-gen} requires to be interpreted
as an identity of terminating Taylor series.

In connection with \eqref{second-gen} some comments are in order.
The generating function $\G$ encodes curvature quantities of a
(pseudo)-Riemannian metric on a manifold $M$. The equality
\eqref{second-gen} relates $\G$ to the volume form of an associated
Poincar\'e-Einstein metric on a space $X =(0,\varepsilon) \times M$
of {\em one more} dimension. The perspective of the AdS/CFT-duality
\cite{W} motivates to refer to this relation as a {\em holographic
duality}. It is important to realize that the variable $r$ plays
fundamentally different roles on both sides of \eqref{second-gen}:
while on the right-hand side it has the {\em geometric} meaning of a
defining function of the boundary $M$ of $X$, on the left-hand side
it is only a {\em formal} variable of a generating function of data
which live {\em on} the boundary $M$.

In addition, it is natural to regard $\G$ as the generating function
of the {\em leading} coefficients $\L_{2N}$ of the residue
polynomials $Q_{2N}^{res}(\lambda)$ (see \eqref{Taylor-q-res}). The
latter polynomials are defined as the constant terms of the
respective residue families $D_{2N}^{res}(\lambda)$. In these terms,
\begin{equation}
\G(r) = - \sum_{N \ge 0} \L_{2N} \frac{r^N}{N!}.
\end{equation}
This interpretation of $\G$ in full generality remains conjectural.
The concept of residue families was introduced in \cite{J-book}.
Their recursive structure and connections with GJMS-operators and
$Q$-curvatures are the origin of all recursive relations discussed
here. The  basic properties of residue families are recalled in
Section \ref{polynomial}.

The paper is organized as follows. In Section \ref{m-def}, we recall
the main properties of GJMS-operators and combine them to the
operators $\M_{2N}$ of order $2N$. We display explicit formula for
$\M_{2N}$ for $N \le 5$ and prove that $\P_{2N}$ has leading part
$\Delta^N$. Section \ref{inf-var} contains the proofs of the
conformal variational formulas for $\P_{2N}$ and $\T_{n/2-1}$. The
conjectural recursive description of GJMS-operators $P_{2N}$ (for
locally conformally flat metrics) is formulated in Section \ref{mc}.
In Section \ref{fourth}, we use the results of Section \ref{inf-var}
to derive a conformally covariant fourth-order power of the
Laplacian (for locally conformally flat metrics). In Sections
\ref{spheres} and \ref{ps-sphere}, we consider the specializations
of Conjecture \ref{formula} to round spheres and pseudo-spheres.
Section \ref{spheres} gives a proof of a refinement for round
spheres. In Section \ref{polynomial}, we explain in which sense the
definition of $\M_{2N}$ is inspired by the $Q$-polynomials of
\cite{J-book}. This sets the background of the formulation of the
conjectural recursive relations for $Q$-curvatures in Section
\ref{Q-rec}. These relations appear in two equivalent forms:
Conjecture \ref{Q-quadrat} and Conjecture \ref{Q-G}. In Section
\ref{Q-rec}, we confirm the general picture for round spheres and
pseudo-spheres. In Section \ref{mix}, we explicate GJMS-operators
for a related class of Riemannian metrics \cite{GL} with terminating
Poincar\'e-Einstein metrics, and confirm Conjecture \ref{formula}
for this class. In Section \ref{n-flat}, we extend the construction
of a conformally covariant fourth power in Section \ref{fourth} to
general metrics, and show that the result confirms a special case of
Conjecture \ref{Haupt}, which  extends Conjecture \ref{formula} to
general metrics. In Section \ref{co}, we collect comments on various
open problems and perspectives. In Section \ref{app} we present
self-contained and detailed proofs of the conformal covariance of
$P_4$ and $P_6$ in the respective critical dimensions and for
general metrics. These proofs illustrate central arguments of the
paper. The reader may start by reading them.

Throughout we use the notation and conventions of \cite{J-book}.
Computer experiments using Mathematica had an important impact on
this work. Such experiments were involved both in the tests of
numerous identities and in the search for the hidden patterns.
Typical instances for the interactions of theoretical and
experimental work are Definition \ref{coefficient}, Lemma \ref{ps}
and Lemma \ref{Q-sphere}. The related programming was done by
Carsten Falk. The material in Section \ref{spheres} emerged from a
discussion with Christian Krattenthaler (Wien). It is a pleasure to
thank him for allowing to present his proof of Theorem \ref{sphere}
in Section \ref{spheres}. Finally, I would like to thank Jesse Alt
(Berlin) and Felipe Leitner (Stuttgart) for comments on the
manuscript.

%%%%%%%%%%%%%%%%%%%%%%%%%%%%%%%%%%%%%%%%%%%%%%%%%%%%%%%%%%%%%%%%%%%%%%%%%%%%%
\section{The operators $\M_{2N}$}
\label{m-def}

We start by recalling the existence of conformally covariant powers
of the Laplacian.

\begin{thm}[\cite{GJMS}]\label{existence} Let $M$ be a manifold
of dimension $n$. For even $n$ and all integers $1 \le N \le \f$,
there exists a natural differential operator $P_{2N}(\cdot)$ of the
form
$$
P_{2N} = \Delta^N + LOT
$$
such that
\begin{equation}\label{conf-cov}
e^{(\f+N)\varphi} \circ P_{2N}(e^{2\varphi}g) = P_{2N}(g) \circ
e^{(\f-N)\varphi}
\end{equation}
for all metrics $g$ and all $\varphi \in C^\infty(M)$. For odd $n$,
such operators exist for all $N \ge 1$.
\end{thm}

More precisely, it is shown in \cite{GJMS} how to derive such
conformally covariant powers of the Laplacian from the powers of the
Laplacian for the Fefferman-Graham ambient metric \cite{FG2}. For
even $n$ and $2N > n$, this construction is obstructed by the
obstructions to the existence of the ambient metric. However, the
non-existence of conformally covariant operators of the form
$\Delta^N + LOT$ for $2N > n$ is a deeper result. In fact, for $N >
\f$ it is impossible to correct $\Delta^N$ by lower order terms so
that the resulting operator satisfies \eqref{conf-cov}. The
non-existence of conformally covariant cubes of the Laplacian on
four-manifolds was discovered in \cite{G-non}. The general
non-existence was established in \cite{GoH}.

On the other hand, for locally conformally flat metrics, all
obstructions vanish, and the construction in \cite{GJMS} yields an
infinite sequence of conformally covariant operators $P_{2N}$ in any
dimension $n \ge 3$. Although in this case the ambient metric is
completely determined by $\Rho$, the complexity of explicit formulas
for the corresponding GJMS-operators increases quickly with their
order.

There are a few exceptional cases, in which simple explicit formulas
are available. On the round sphere $\S^n$, GJMS-operators are
intertwining operators for spherical principal series
representations. Hence they can be derived from the standard
Knapp-Stein intertwining operators. This yields the formula
\begin{equation}\label{product}
P_{2N} = \prod_{j=\f}^{\f+N-1} (\Delta \!-\! j(n\!-\!1\!-\!j)).
\end{equation}
The product formula \eqref{product} extends to Einstein metrics in
the form
\begin{equation}\label{einstein}
P_{2N} = \prod_{j=\f}^{\f+N-1} \left(
\Delta\!-\!\frac{j(n\!-\!1\!-\!j)}{n(n\!-\!1)} \tau \right),
\end{equation}
where $\tau$ is the constant scalar curvature of the Einstein
metric. For details see \cite{bran-2}, \cite{G-ein}, \cite{G-power},
\cite{FG2}, and \cite{J-book}. These examples will serve as basic
test cases of general statements.

In the following, we shall often distinguish (for even $n$) between
the {\em critical} GJMS-operator $P_n$ and the {\em subcritical}
GJMS-operators $P_{2N}$, $2N<n$.

By relating the operators $P_{2N}$ to scattering theory for
Poincar\'e-Einstein metrics, Graham and Zworski \cite{GZ} proved
that all $P_{2N}$ are formally self-adjoint.

Branson \cite{bran-2} used the constant term of $P_{2N}$ to define
the scalar curvature quantity $Q_{2N}$ through the formula
\begin{equation}\label{q-curv}
P_{2N}(1) = (-1)^N \left(\f-N\right) Q_{2N};
\end{equation}
note that the sign $(-1)^N$ is caused by our convention that
$-\Delta$ is the non-negative Laplacian. For even $n$, this defines
$Q_{2N}$ for $2N < n$, and we shall refer to these functions as to
the {\em subcritical} $Q$-curvatures. $Q_{2N}$ is a curvature
quantity of order $2N$, i.e., its definition involves $2N$
derivatives of the metric. For even $n$, the {\em critical}
$Q$-curvature $Q_n$ arises from its subcritical relatives of order
$n$ (but in dimension $>n$) by the "limit" $\text{dimension} \to n$.
For $Q_4$ as in \eqref{q4-gen}, this just means to set $n=4$.

Similarly, as for $Q_2$ and $Q_4$, the critical $Q$-curvature $Q_n$
satisfies the fundamental linear transformation law
\begin{equation}\label{fundamental}
e^{n\varphi} Q_n(e^{2\varphi}g) = Q_n(g) + (-1)^\f P_n(g)(\varphi),
\end{equation}
which involves the critical GJMS-operator $P_n$. It shows the
remarkable fact that the operator $P_n$ is completely determined by
the scalar curvature quantity $Q_n$. In the subcritical cases, the
non-constant part $P_{2N}^0$ of $P_{2N}$ is determined by the
conformal variation of $Q_{2N}$ (see \eqref{sub-det}). However, the
subcritical $Q$-curvatures do not obey a linear conformal
transformation law.

Now the operators $P_{2N}$ give rise to a sequence of operators
$\M_{2N}$, $N \ge 1$. As for $P_{2N}$, this sequences is infinite in
odd dimensions and possibly obstructed at $2N=n$ in even dimension
$n$. These restrictions are in force throughout and are suppressed
in the following.

We introduce some notation. A sequence $I = (I_1,\dots,I_r)$ of
integers $I_j \ge 1$ will be regarded as a composition of the sum
$\vert I \vert = I_1 + I_2 + \cdots + I_r$. Compositions are
partitions in which the order of the summands is considered. $|I|$
will be called the size of $I$. We set
$$
P_{2I} = P_{2I_1} \circ \cdots \circ P_{2I_r}.
$$

\begin{defn}\label{coefficient} For $N \ge 1$, let
\begin{equation}\label{M-sum}
\M_{2N} = \sum_{|I|=N} m_I P_{2I}
\end{equation}
with
\begin{equation}\label{m-form}
m_I = -(-1)^r |I|! (|I|-1)! \prod_{j=1}^r \frac{1}{I_j! (I_j-1)!}
\prod_{j=1}^{r-1} \frac{1}{I_j + I_{j+1}}.
\end{equation}
Then $m_{(N)} = 1$, and we define the primary part $\P_{2N}$ by the
decomposition
\begin{equation}\label{deco}
\M_{2N} = P_{2N} - \P_{2N}.
\end{equation}
\end{defn}

Note that
\begin{equation*}
m_I =  -(-1)^r \begin{pmatrix} |I| \\ I_1, \dots, I_r \end{pmatrix}
\begin{pmatrix} |I|\!-\!r \\ I_1\!-\!1, \dots, I_r\!-\!1 \end{pmatrix}
\frac{(N\!-\!1) \cdots (N\!-\!r\!+\!1)}{\prod_{j=1}^{r-1} (I_j +
I_{j+1})}
\end{equation*}
in terms of multinomial coefficients. Although it will not be
important in the sequel, it would be interesting to know whether all
$m_I$ are integers and whether $m_I$ has a combinatorial meaning.
For compositions $I$ with two entries, we easily find
\begin{equation}\label{m-double}
m_{(I_1,I_2)} = - \binom{N\!-\!1}{I_1}\binom{N\!-\!1}{I_2} \in \Z,
\; N = I_1+I_2.
\end{equation}

The sum in \eqref{M-sum} runs over all compositions $I$ of size
$|I|=N$. It contains $2^{N-1}$ terms. More precisely, there are
exactly $\binom{N-1}{r-1}$ terms with $r$ factors. For more details
on compositions see \cite{A}, Chapter 4.

Since $m_I \ne 0$ for all $I$, each possible composition of
GJMS-operators to an operator of order $2N$ contributes
non-trivially to the sum \eqref{M-sum}. Obstructions to the
existence of $\M_{2N}$ arise only through obstructions to the
existence of the GJMS-operators.

In general, GJMS-operators do not commute, and the coefficients
$m_I$ depend on the ordering of the entries of the composition $I$.
This is the reason for the consideration of compositions instead of
partitions. However, we observe

\begin{corr}\label{self} $m_I = m_{I^{-1}}$ for all $I$,
where $I^{-1} = (I_r,\dots,I_1)$ denotes the reversed (or inverse)
composition of $I=(I_1,\dots,I_r)$. In particular, all $\M_{2N}$ are
self-adjoint.
\end{corr}

\begin{proof} The claimed symmetry of the coefficients is obvious
from \eqref{m-form}. Therefore, the self-adjointness of all
GJMS-operators implies the self-adjointness of all $\M_{2N}$.
\end{proof}

We display the first few operators $\M_{2N}$. The first two cases
are very simple.

\begin{ex}\label{m24} $\M_2 = P_2$ and $\M_4 = P_4 - P_2^2$. The
corresponding primary parts are
\begin{equation}\label{prim-2}
\P_2 = 0 \quad \mbox{and} \quad \P_4 = P_2^2.
\end{equation}
\end{ex}

The following two cases will play a substantial role in what
follows.

\begin{ex}\label{m6}
$$
\M_6 = P_6 - \P_6
$$
with the primary part
\begin{equation}\label{prim-4}
\P_6 = 2(P_2 P_4 + P_4 P_2) - 3 P_2^3.
\end{equation}
\end{ex}

\begin{ex}\label{m8}
$$
\M_8 = P_8 - \P_8
$$
with the primary part
\begin{equation}\label{prim-8}
\P_8 = (3 P_2 P_6 + 3 P_6 P_2 + 9 P_4^2) - (12 P_2^2 P_4 + 12 P_4
P_2^2 + 8 P_2 P_4 P_2) + 18 P_2^4.
\end{equation}
The sum contains $7$ terms.
\end{ex}

The following formula illustrates the exponentially increasing
complexity of the situation.

\begin{ex}\label{m10}
$$
\M_{10} = P_{10}- \P_{10}
$$
with the primary part
\begin{multline}\label{prim-10}
\P_{10} = (4 P_2 P_8 + 4 P_8 P_2 + 24 P_4 P_6 + 24 P_6 P_4) \\
- (60 P_2 P_4^2 + 60 P_4^2 P_2 + 30 P_2^2 P_6 + 30 P_6 P_2^2 + 15
P_2 P_6 P_2 + 80 P_4 P_2 P_4) \\ + (120 P_2^3 P_4 + 120 P_4 P_2^3 +
80 P_2^2 P_4 P_2 + 80 P_2 P_4 P_2^2) - 180 P_2^5.
\end{multline}
The sum contains $15$ terms.
\end{ex}

The operator $\P_{2N}$ will play the role of a primary part of the
GJMS-operator $P_{2N}$ in the sense that it differs from $P_{2N}$
only by a second-order operator. The following result is the minimal
requirement in order to qualify $\P_{2N}$ for that role.

\begin{lemm}\label{LPP} For all $N \ge 2$, the operator $\P_{2N}$ is
of the form $\Delta^N + LOT$.
\end{lemm}

\begin{proof} The assertion is equivalent to
\begin{equation}\label{sum-0}
\sum_{|I|=N} m_I = 0.
\end{equation}
\eqref{sum-0} follows from the stronger relations
\begin{equation}\label{strong}
\sum_{J,\; a +|J|=N} m_{(a,J)} = (-1)^{N-a} \binom{N\!-\!1}{a\!-\!1}
\end{equation}
for  $1 \le a \le N$. In fact, \eqref{strong} implies
\begin{multline*}
(-1)^{N-1} \sum_{|I|=N} m_I = (-1)^{N-1} \sum_{a=1}^N
\left( \sum_{J, \; a+|J|=N} m_{(a,J)} \right) \\
= \binom{N\!-\!1}{0} - \binom{N\!-\!1}{1} \pm \cdots + (-1)^{N-1}
\binom{N\!-\!1}{N\!-\!1} = 0.
\end{multline*}
In order to prove \eqref{strong}, we write
$$
\sum_{J,\; a+|J|=N} m_{(a,J)} = \sum_{b,K, \; a+b+|K|=N} m_{(a,b,K)},
$$
and note that \eqref{m-form} implies
\begin{equation}\label{bino}
m_{(a,b,K)} = - \frac{1}{a+b} \binom{N}{a}^2 \, \frac{a
(N\!-\!a)}{N} m_{(b,K)}.
\end{equation}
Hence
\begin{equation}\label{Ind}
\sum_{J, \; a+|J|=N} m_{(a,J)} = - \binom{N}{a}^2 \, \frac{a
(N\!-\!a)}{N} \sum_{b=1}^{N-a} \frac{1}{a+b} \sum_{K, \; a+b+|K|=N}
m_{(b,K)}.
\end{equation}
We use \eqref{Ind} to prove \eqref{strong} by induction on $N$.
Suppose we have already proved \eqref{strong} up to $N-1$. Then the
right-hand side of \eqref{Ind} equals
$$
-\binom{N}{a}^2 \, \frac{a (N\!-\!a)}{N} \sum_{b=1}^{N-a} \frac{1}{a+b}
(-1)^{N-a-b} \binom{N\!-\!a\!-\!1}{b\!-\!1}.
$$
Thus, it suffices to verify that
\begin{equation}\label{kernel}
- \binom{N\!-\!1}{a\!-\!1} = \binom{N}{a}^2 \frac{a(N\!-\!a)}{N}
\sum_{b=1}^{N-a} \frac{1}{a\!+\!b} (-1)^b
\binom{N\!-\!a\!-\!1}{b\!-\!1}.
\end{equation}
To this end, we apply the identities
\begin{equation}\label{beta}
\sum_{j=0}^N \frac{1}{j+M} (-1)^j \binom{N}{j} = B(M,N+1), \; M \ge 1
\end{equation}
which follow from the formula
$$
B(M,N\!+\!1) = \int_0^1 x^{M-1} (1-x)^N dx
$$
for the Beta function by expanding $(1-x)^N$ as a polynomial in $x$
and integrating term by term. \eqref{beta} implies
$$
\sum_{b=1}^{N-a} \frac{1}{a+b} (-1)^b \binom{N\!-\!a\!-\!1}{b\!-\!1}
= - B(a\!+\!1,N\!-\!a) = - \frac{a!(N\!-\!a\!-\!1)!}{(N\!-\!1)!}.
$$
Now a calculation shows that
$$
\binom{N\!-\!1}{a\!-\!1} \Big/ \binom{N}{a}^2 \frac{a(N\!-\!a)}{N} =
\frac{a!(N\!-\!a\!-\!1)!}{N!}.
$$
This proves \eqref{kernel}.
\end{proof}

The relation \eqref{strong} will be substantially refined in Section \ref{spheres}.

%%%%%%%%%%%%%%%%%%%%%%%%%%%%%%%%%%%%%%%%%%%%%%%%%%%%%%%%%%%%%%%%%%%%%%%%%%%%%
\section{Conformal variational formulas}\label{inf-var}

In the present section, we prove conformal variational formulas for
the operators
$$
\M_{2N} = P_{2N} - \P_{2N} \quad \mbox{and} \quad \T_{n/2-1},
$$
where
\begin{equation}\label{second}
\T_{N} \st \delta (\Rho^N \# d).
\end{equation}
In \eqref{second}, the notation does not distinguish between the
symmetric bilinear form $\Rho$ and the induced linear operator on
$TM$. We use this convention throughout.

The first result concerns the primary parts $\P_{2N}$.

\begin{thm}\label{c-cv} On $M^n$,
\begin{multline}\label{c-cv-gen}
(d/dt)|_0 \big(e^{(\f+N)t\varphi} \P_{2N}(e^{2t\varphi}g)
e^{-(\f-N)t\varphi} \big) \\= \sum_{j=1}^{N-1}
\binom{N\!-\!1}{j\!-\!1}^2 (N\!-\!j)
\left[\M_{2j}(g),[\M_{2N-2j}(g),\varphi]\right].
\end{multline}
Here $\varphi$ is regarded as a multiplication operator.
\end{thm}

\eqref{c-cv-gen} holds true whenever both sides are defined. Thus,
for even $n$ and general metrics, we assume that $2N \le n$.

In the critical case, Theorem \ref{c-cv} states that
\begin{equation*}
(d/dt)|_0 \left(e^{n t\varphi} \P_n(e^{2t\varphi}g) \right) =
\sum_{j=1}^{\f-1} \binom{\f\!-\!1}{j\!-\!1}^2 (\f\!-\!j)
\left[\M_{2j}(g),\left[\M_{n-2j}(g),\varphi\right]\right].
\end{equation*}

\begin{proof} The proof rests on the transformation laws
\eqref{conf-cov}. The right-hand side of \eqref{c-cv-gen} is a
weighted sum of terms of the form
\begin{equation*}
\M_{2j} \circ \M_{2N-2j} \circ \varphi - \M_{2j} \circ \varphi \circ
\M_{2N-2j} - \M_{2N-2j} \circ \varphi \circ \M_{2j} + \varphi \circ
\M_{2N-2j} \circ \M_{2j}.
\end{equation*}
On the left-hand side of \eqref{c-cv-gen}, the term $P_{2I}$ with
$I=(I_1,I_2,\dots,I_r)$ in the sum $\P_{2N}$ induces a constant
multiple of the contribution
\begin{multline}\label{lhs}
(N\!-\!I_1) \varphi \circ P_{2I} - (I_1\!+\!I_2) P_{2I_1} \circ
\varphi \circ P_{2I_2} \cdots P_{2I_r} \\
- \cdots - (I_{r-1}\!+\!I_r) P_{2I_1} \cdots P_{2I_{r-1}} \circ
\varphi \circ P_{2I_r} + (N\!-\!I_r) P_{2I} \circ \varphi.
\end{multline}
For the term $P_{2I} \circ \varphi$, the claim is
\begin{multline}\label{c-last}
-(N\!-\!I_r) m_I = \binom{N\!-\!1}{I_1\!-\!1}^2
(N\!-\!I_1) \; m_{(I_1)} m_{(I_2,\dots,I_r)} \\
+ \binom{N\!-\!1}{I_1\!+\!I_2\!-\!1}^2
(N\!-\!I_1\!-\!I_2) \; m_{(I_1,I_2)} m_{(I_3,\dots,I_r)} \\
+ \cdots + \binom{N\!-\!1}{I_1\!+\!I_2\!+ \cdots +\!I_{r-1}\!-\!1}^2
(N\!-\!I_1\!-\!I_2-\cdots-I_{r-1}) \; m_{(I_1,I_2,\dots,I_{r-1})}
m_{(I_r)}.
\end{multline}
In order to prove this identity, we use the explicit formula for the
coefficients $m_I$ (see \eqref{m-form}) to write the terms in the
sum as multiples of $m_I$. We find
\begin{multline}\label{comb}
-\frac{1}{N} \Big[ I_1(I_1+I_2) + (I_1+I_2)(I_2+I_3) +
(I_1+I_2+I_3)(I_3+I_4) \\ + \cdots +
(I_1+I_2+\cdots+I_{r-1})(I_{r-1}+I_r)\Big] m_I.
\end{multline}
Now the relation
\begin{equation*}
I_1(I_1+I_2) + \cdots + (I_1+I_2+\cdots+I_{r-1})(I_{r-1}+I_r)  =
(I_1+\cdots+I_r)(I_1+\cdots+I_{r-1})
\end{equation*}
(which follows by induction) implies that in \eqref{comb} the sum in
brackets equals $N(N-I_r)$ if $|I|=N$. Thus, \eqref{comb} equals
$-(N-I_r)m_I$. This proves the assertion.

Next, for the term $\varphi \circ P_{2I}$ with $|I|=N$, the claim is
\begin{multline*}
-(N\!-\!I_1) m_I = \binom{N\!-\!1}{I_2\!+\cdots+\!I_r\!-\!1}^2
(N\!-\!I_2\!-\cdots-I_r) \; m_{(I_1)} m_{(I_2,I_2,\dots,I_{r})} \\
+ \cdots + \binom{N\!-\!1}{I_r\!-\!1}^2 (N\!-\!I_r) \;
m_{(I_1,\dots,I_{r-1})} m_{(I_r)}.
\end{multline*}
This identity follows by applying \eqref{c-last} to the inverse
composition $I^{-1}$ of $I$ and using the relations $m_{I^{-1}} =
m_I$ for all compositions $I$ (see Corollary \ref{self}).

It remains to prove the corresponding identities for the
coefficients of the terms
$$
P_{2I_1} \cdots P_{2I_a} \circ \varphi \circ P_{2I_{a+1}} \dots
P_{2I_r}.
$$
In that case, the claim is
\begin{multline*}
-(I_a+I_{a+1}) m_I = \Big[
\binom{N\!-\!1}{I_1\!+\cdots+\!I_a\!-\!1}^2
(N\!-\!I_1\!-\cdots-I_a) \\
+ \binom{N\!-\!1}{I_{a+1}\!+\cdots+\!I_r\!-\!1}^2
(N\!-\!I_{a+1}\!-\cdots-I_r) \Big] m_{(I_1,\dots,I_a)}
m_{(I_{a+1},\dots,I_r)}.
\end{multline*}
By \eqref{m-form}, the right-hand side reduces to
$$
-\frac{1}{N} (I_a+I_{a+1}) [(I_1+\cdots+I_a) + (N-I_1-\cdots-I_a)]
m_I,
$$
i.e., to $-(I_a + I_{a+1)}) m_I$. This completes the proof.
\end{proof}

\begin{corr}\label{nice} In the situation of Theorem \ref{c-cv},
\begin{equation*}
(d/dt)|_0 \left(e^{(\f+N)t\varphi} \V_{2N}(e^{2t\varphi}g)
e^{-(\f-N)t\varphi} \right) = \sum_{j=1}^{N-1} \frac{1}{N-j}
\left[\V_{2j}(g),[\V_{2N-2j}(g),\varphi]\right],
\end{equation*}
where
\begin{equation}\label{V-def}
\V_{2k} \st - \frac{\M_{2k}}{(k-1)! (k-1)!}.
\end{equation}
\end{corr}

\begin{proof} By $\M_{2N} = P_{2N} - \P_{2N}$ and the conformal
covariance \eqref{conf-cov} of $P_{2N}$, the assertion is equivalent
to Theorem \ref{c-cv}. \end{proof}

The self-adjointness of $\V_{2N}$ (Corollary \ref{self}) implies the
self-adjointness of the conformal variation
$$
(d/dt)|_0 \big(e^{(\f+N)t\varphi} \V_{2N}(e^{2t\varphi}g)
e^{-(\f-N)t\varphi} \big).
$$
In fact, for $u,v \in C^\infty(M)$ with compact support,
\begin{align*}
& \int_M (d/dt)|_0 \left( e^{(\f+N)t\varphi}
\V_{2N}(e^{2t\varphi}g)(e^{-(\f-N)t\varphi} u) \right) \bar{v} \, vol(g) \\
& = (d/dt)|_0 \int_M e^{(-\f+N)t\varphi} \V_{2N}(e^{2t\varphi}g)
(e^{-(\f-N)t\varphi} u) \bar{v} \, vol(e^{2t\varphi} g) \\
& = (d/dt)|_0 \int_M u \overline{e^{-(\f-N)t\varphi}
\V_{2N}(e^{2t\varphi}g)(e^{-(\f-N) t\varphi} v)} \, vol(e^{2t\varphi} g) \\
& = \int_M u \overline{ (d/dt)|_0 \left( e^{(\f+N)t\varphi}
\V_{2N}(e^{2t\varphi}g) (e^{-(\f-N)t\varphi})v \right)} \, vol(g).
\end{align*}
Corollary \ref{nice} confirms this observation for $\V_{2N}$ by
using the self-adjointness of all lower order $\V_{2M}$, $M < N$.

The second conformal variational formula concerns the operator
$\T_{n/2-1}$ on manifolds $M^n$ of even dimension. Note that $\T_0 =
-\Delta$.

\begin{thm}\label{comm-2} For a locally conformally flat metric $g$,
\begin{equation}\label{var-t}
n (d/dt)|_0 \left(e^{nt\varphi} \T_{\f-1}(e^{2t\varphi}g)\right) =
\sum_{j=1}^{\f-1} j [\T_{j-1}(g),[\T_{\f-1-j}(g),\varphi]]^0.
\end{equation}
\end{thm}

The proof of Theorem \ref{comm-2} will also show that for general
metrics both sides of \eqref{var-t} differ by a second-order
operator the main part of which is given by the sum of the terms
\begin{equation}\label{non-flat}
4k (\Rho^{k-1})_a^t (\Rho^{l-1})^r_c (\Rho^{\f-1-k-l})^s_b \Y_{trs}
\varphi^c \Hess^{ab}(u)
\end{equation}
for all integers $l,k \ge 1$ such that $l+k \le \f-1$. Here $\Y$
denotes the Cotton tensor
\begin{equation}\label{cotton}
\Y(X,Y,Z) = \nabla_X(\Rho)(Y,Z) - \nabla_Y(\Rho)(X,Z).
\end{equation}
We recall that $\Y$ vanishes if the Weyl tensor $\C$ vanishes. This
result will be used in Section \ref{n-flat}.

\begin{proof} The assertion relates two self-adjoint second-order
differential operators which annihilate constants. Hence it suffices
to prove that the main parts of both operators coincide. Now a
calculation (using $\delta(T \# d) = - (T,\Hess) + (\delta(T),d)$
for any symmetric bilinear form $T$) shows that the main part of the
operator
$$
\left[\T_p,[\T_q,\varphi]\right], \; p, q \ge 0
$$
is of the form
\begin{multline}\label{d-comm}
4 (\Rho^p)_{ij} (\Rho^q)^{rs} \Hess^i_r(\varphi) \Hess^j_s(\varphi)
\\ + 4 (\Rho^p)_{ij} \nabla^i (\Rho^q)^{rs} \varphi_r \Hess^j_s(u) -
2 (\Rho^q)_{ij} \nabla^j(\Rho^p)_{rs} \varphi^i \Hess^{rs}(u).
\end{multline}
Hence the right-hand side of \eqref{var-t} equals
$$
4 \sum_{k=1}^{\f-1} k (\Rho^{k-1})_{ij} (\Rho^{\f-1-k})^{rs}
\Hess_r^j(\varphi) \Hess_s^j(u),
$$
up to terms with first order derivatives of $\Rho$. The latter sum
simplifies to
\begin{equation}\label{no-deriv}
n \sum_{k=1}^{\f-1} (\Rho^{k-1})_{ij} (\Rho^{\f-1-k})^{rs}
\Hess_r^i(\varphi) \Hess_s^j(u).
\end{equation}
We compare \eqref{no-deriv} with the main part of the left-hand side
of \eqref{var-t}. It is given by
\begin{multline*}
-\sum_{k=1}^{\f-1} (d/dt)|_0 \left((\Rho^{k-1})_{ir} (\Rho^r_s \!-\!
t \Hess^r_s(\varphi)) (\Rho^{\f-1-k})_j^s \right) \Hess^{ij}(u) \\
= \sum_{k=1}^{\f-1} (\Rho^{k-1})_{ij} (\Rho^{\f-1-k})^{rs}
\Hess^j_s(\varphi) \Hess^i_r(u).
\end{multline*}
Therefore, it only remains to prove that, in the locally conformally
flat case, i.e., if $\C=0$, the terms with derivatives of $\Rho$
cancel. By \eqref{d-comm}, these contribute
\begin{multline*}
4 \sum_{k=1}^{\f-1} k (\Rho^{k-1})_{ia} \nabla^i
(\Rho^{\f-1-k})^{cb} \varphi^c \Hess^{ab}(u) \\[-4mm] - 2
\sum_{k=1}^{\f-1} k (\Rho^{\f-1-k})_{cj} \nabla^j(\Rho^{k-1})_{ab}
\varphi^c \Hess^{ab}(u).
\end{multline*}
Reordering the second sum yields
\begin{multline*}
\left( 4 \sum_{k=1}^{\f-2} k \Rho^{k-1}_{ia}
\nabla^i(\Rho^{\f-1-k})^{cb} - 2 \sum_{k=1}^{\f-2}
\left(\f\!-\!k\right) (\Rho^{k-1})_{ic} \nabla^i
(\Rho^{\f-1-k})^{ab} \right) \varphi^c \Hess^{ab}(u).
\end{multline*}
Now we group the terms in this sum as follows. The product rule
turns the derivatives of powers of $\Rho$ into sums of products
which contain one derivative of $\Rho$. We match the resulting sum
in the $k^{\text{th}}$ term in the first sum with the sum of the
respective $k^{\text{th}}$ terms in the individual contributions in
the second sum. This gives
\begin{multline*}
\sum_{k=1}^{\f-2} \Big\{ \sum_{l=1}^{\f-1-k} 4k (\Rho^{k-1})_{ia}
(\Rho^{l-1})_{cr} \nabla^i(\Rho)^r_s (\Rho^{\f-1-k-l})^s_b \\
- 2 \sum_{l=1}^{\f-1-k} (\f-l) (\Rho^{l-1})_{ic} (\Rho^{k-1})_{ar}
\nabla^i (\Rho)_s^r (\Rho^{\f-1-k-l})^s_b \Big\}.
\end{multline*}
Interchanging the roles of $i$ and $r$ in the second sum, yields
\begin{multline*}
\sum_{k=1}^{\f-2} \Big\{ \sum_{l=1}^{\f-1-k} 4k (\Rho^{k-1})_{ia}
(\Rho^{l-1})_{cr} \nabla^i(\Rho)^r_s (\Rho^{\f-1-k-l})^s_b \\
- 2 \sum_{l=1}^{\f-1-k} (\f-l) (\Rho^{k-1})_{ai} (\Rho^{l-1})_{rc}
\nabla^r (\Rho)_s^i (\Rho^{\f-1-k-l})^s_b \Big\}.
\end{multline*}
We rewrite this sum as
\begin{multline*}
\sum_{k=1}^{\f-2} \Big\{ \sum_{l=1}^{\f-1-k} 4k (\Rho^{k-1})_{ia}
(\Rho^{l-1})_{cr} \left[ \nabla^i(\Rho)_s^r - \nabla^r(\Rho)_s^i
\right] (\Rho^{\f-1-k-l})^s_b \\
- 2 \sum_{l=1}^{\f-1-k} \left(\f-l-2k\right) (\Rho^{k-1})_{ia}
(\Rho^{l-1})_{cr} \nabla^r (\Rho)_s^i (\Rho^{\f-1-k-l})^s_b \Big\},
\end{multline*}
i.e., as
\begin{multline*}
4 \sum_{k=1}^{\f-2} \sum_{l=1}^{\f-1-k} k (\Rho^{k-1})_a^i
(\Rho^{l-1})_c^r \Y_{irs} (\Rho^{\f-1-k-l})^s_b \\
- 2 \sum_{k=1}^{\f-2} \sum_{l=1}^{\f-1-k} \left(\f-l-2k\right)
(\Rho^{k-1})_a^i (\Rho^{l-1})_c^r \nabla^r (\Rho)_{is}
(\Rho^{\f-1-k-l})^s_b.
\end{multline*}
In the latter double sum we match, for given $l$, the terms for $k$
and $k'=\f-k-l$ (note that both numbers coincide iff the coefficient
$\f-l-2k$ vanishes). These terms are given by
\begin{multline*}
-2 \left(\f-l-2k\right) (\Rho^{k-1})_a^i (\Rho^{l-1})_c^r
(\Rho^{\f-1-k-l})^s_b \nabla^r (\Rho)_{is} \\
+ 2\left(\f-l-2k\right) (\Rho^{\f-1-k-l})^i_a (\Rho^{l-1})_c^r
(\Rho^{k-1})_b^i \nabla^r (\Rho)_{is}.
\end{multline*}
By summation against $\Hess^{ab}(u)$, these sums vanish. This yields
the explicit formula \eqref{non-flat} for those contributions to the
main part of the right-hand side of \eqref{var-t} which contain
derivatives of $\Rho$.
\end{proof}

The same argument as for $\M_{2N}$ shows that the conformal
variation of $\T_{n/2-1}$ is self-adjoint. Using the
self-adjointness of all $\T_{2M}$, $M\le \f-1$, Theorem \ref{comm-2}
confirms this observation for locally conformally flat metrics.

%%%%%%%%%%%%%%%%%%%%%%%%%%%%%%%%%%%%%%%%%%%%%%%%%%%%%%%%%%%%%%%%%%%%%%%%%
\section{Universal recursive formulas for GJMS-operators}\label{mc}

In the present section, we work in the locally conformally flat
category, i.e., we assume that the Weyl tensor $\C$ vanishes. In
this case, the Fefferman-Graham ambient metric terminates at the
third term \cite{FG2}, and the existence of GJMS-operators is not
obstructed. For comments on the general case we refer to Section
\ref{n-flat}.

The following conjecture states a recursive formula for
GJMS-operators $P_{2N}$.

\begin{conj}[{\bf Universal recursive formulas for GJMS-operators}]
\label{formula} Let $(M,g)$ be a locally conformally flat Riemannian
manifold of dimension $n \ge 3$. Then for $N \ge 1$,
\begin{equation}\label{universal}
\M_{2N}^0 = - c_N \delta (\Rho^{N-1}\# d),
\end{equation}
where $c_N = 2^{N-1} N!(N-1)!$. Here $[\cdot]^0$ denotes the
non-constant part of the respective operator, and $\delta$ is the
negative divergence.
\end{conj}

Some remarks are in order. First of all, since  $\M_{2N}$ is of the
form $P_{2N} - \P_{2N}$, the relation \eqref{universal} is
equivalent to
$$
P_{2N}^0 = \P_{2N}^0 - c_N \delta (\Rho^{N-1}\# d) = \P_{2N}^0 - c_N
\T_{N-1}.
$$
This formula presents the non-constant part of $P_{2N}$ as a linear
combination of {\em all} products of order $2N$ which can be formed
by using lower order GJMS-operators (up to the contribution
$\T_{N-1}$).

For $\r^n$ with the Euclidean metric, $P_{2N}$ coincides with $\Delta^N$
and $\M_{2N} = 0$ follows from the summation formula \eqref{sum-0}.

\eqref{universal} combined with
$$
P_{2N}(1) = (-1)^N \left(\f-N\right) Q_{2N}
$$
yields a recursive formula for $P_{2N}$. Resolving the recursion
leads to a formula for $P_{2N}$ in terms of the $Q$-curvatures
$Q_{2N}, \dots, Q_2$ and the powers of the Schouten tensor $\Rho$.
The critical $P_n$ has the special property that it only depends on
the lower order $Q$-curvatures $Q_{n-2},\dots,Q_2$. In turn, using
recursive relations for $Q$-curvatures in terms of lower order
GJMS-operators and lower order $Q$-curvatures (see Conjecture
\ref{Q-quadrat}) further reduces $P_{2N}$ step by step to lower
order constructions. This method generates formulas for
GJMS-operators in terms of the Schouten tensor $\Rho$ and its
derivatives (under the assumption $\C=0$).

\eqref{universal} is called {\em universal} since the coefficients
of $\M_{2N}$ do not depend on the dimension. In other words, this
way of writing the non-constant part of the critical GJMS-operator
literally extends to the non-critical cases.

The recursive formula \eqref{universal} is expected to extend also
to the pseudo-Riemannian case (see Section \ref{ps-sphere} for the
discussion of a special case).

It is natural to summarize the relations \eqref{universal} in terms
of generating functions as
\begin{equation}\label{gf-cf-m}
\sum_{N \ge 1} \V_{2N}^0 \left(\frac{r^2}{4}\right)^{N-1} = \delta
((1 - r^2/2 \,\Rho)^{-2} \#d),
\end{equation}
where $\V_{2N}$ is defined by \eqref{V-def}. A natural
generalization of \eqref{gf-cf-m} for general metrics $g$ will be
discussed in Section \ref{n-flat}.

Conjecture \ref{formula} is supported by the special cases $N=1$,
$N=2$ and $N=3$. For $N=1$, \eqref{universal} just states the
obvious relation $P_2^0 = \Delta$. The following result follows from
\eqref{pan} by a direct calculation.

\begin{thm}\label{M4} On manifolds of dimension $n \ge 3$, the
Paneitz operator $P_4$ is given by
\begin{equation}\label{p4}
P_4^0 = (P_2^2)^0 - 4 \delta (\Rho \# d)
\end{equation}
and
$$
P_4(1) = \left(\f\!-\!2\right) Q_4, \quad Q_4 = \f \J^2 - 4 |\Rho|^2
- \Delta \J.
$$
\end{thm}

In Section \ref{app-pan}, we shall derive the conformal covariance
of $P_4$ in the critical case $n=4$ directly from \eqref{p4}.

\begin{thm}[\cite{J-book}, Corollary 6.12.2]\label{M6} On manifolds of
dimension $n \ge 6$, the GJMS-operator $P_6$ is given by
\begin{equation}\label{p6-general}
P_6^0 = \big[2(P_2 P_4 + P_4 P_2) - 3 P_2^3\big]^0 - 48 \delta
(\Rho^2 \# d) - \frac{16}{n\!-\!4} \delta (\B \# d)
\end{equation}
and
$$
P_6(1) = -\left(\f\!-\!3\right) Q_6.
$$
\end{thm}

Here the tensor
$$
\B_{ij} = \Delta(\Rho)_{ij} - \nabla^k \nabla_j(\Rho)_{ik} +
\Rho^{kl} \C_{kijl}
$$
generalizes the Bach tensor in dimension $4$. The Bach tensor term
in \eqref{p6-general} obstructs the existence of $P_6$ in dimension
$n=4$.

In the locally conformally flat case, \eqref{p6-general} simplifies
to
\begin{equation}\label{p6-flat}
P_6^0 = \big[2(P_2 P_4 + P_4 P_2) - 3 P_2^3\big]^0 - 48 \delta
(\Rho^2 \# d).
\end{equation}
Obviously, this formula is a special case of Conjecture
\ref{formula}.

In \cite{J-book}, we used the relation
$$
-P_6^0(g)(u) = (d/dt)|_0 \left(e^{6tu} Q_6(e^{2tu}g)\right)
$$
to derive \eqref{p6-general} from the explicit formula
\begin{equation}\label{q6}
Q_6 = \left[-2P_2(Q_4) + 2 P_4(Q_2) - 3 P_2^2(Q_2)\right] - 6(Q_4 +
P_2(Q_2)) Q_2 - 3! 2! 2^5 v_6
\end{equation}
for $Q_6$. Here
\begin{equation}\label{v6}
v_6 = - \frac{1}{8} \tr (\wedge^3 \Rho) - \frac{1}{24(n\!-\!4)}
(\B,\Rho).
\end{equation}

The discussion in Section \ref{polynomial} will show that \eqref{q6}
should be regarded as a special case of the recursive formula
\eqref{Q-gen-rec} for $Q$-curvatures in terms of the leading
coefficients of the $Q$-polynomials $Q_{2N}^{res}(\lambda)$.

In Section \ref{inv-4-6}, we present an alternative proof of the
conformal covariance of $P_6$ in the critical dimension $n=6$. It
illustrates the argument provided by Theorem \ref{inv}. In the
special case of $P_6$, it derives Theorem \ref{M6} in dimension
$n=6$ from Theorem \ref{M4} and $P_2^0 = \Delta$ by using Theorem
\ref{c-cv} and Theorem \ref{comm-2}.

\begin{thm}\label{inv} For a locally conformally flat metric $g$, the
relations
\begin{equation}\label{sub-rep}
\M_{2N}^0(g) = -c_N \T_{N-1}(g) \quad \mbox{for all} \quad
N=1,\dots,\f-1
\end{equation}
imply
\begin{equation}\label{inf-pn}
(d/dt)|_0 \left(e^{nt\varphi} (\P_n^0 - c_\f
\T_{\f-1})(e^{2t\varphi}g)\right) = 0.
\end{equation}
\end{thm}

\begin{proof} Theorem \ref{c-cv} (in the critical case) implies
\begin{equation*}
(d/dt)|_0 \left(e^{n t\varphi} \P_n(e^{2t\varphi}g) \right) =
\sum_{j=1}^{\f-1} \binom{\f\!-\!1}{j\!-\!1}^2 \left(\f\!-\!j\right)
\left[\M_{2j}(g),\left[\M_{n-2j}(g),\varphi\right]\right].
\end{equation*}
The commutators with $\varphi$ do not depend on the constant terms
of the respective operators. Moreover, since these commutators are
of first order, their commutators with any operator contain the
constant term of the latter only in their constant terms. It follows
that
\begin{equation*}
(d/dt)|_0 \left(e^{n t\varphi} \P_n^0(e^{2t\varphi}g) \right) =
\sum_{j=1}^{\f-1} \binom{\f\!-\!1}{j\!-\!1}^2 \left(\f\!-\!j\right)
\left[\M^0_{2j}(g),\left[\M^0_{n-2j}(g),\varphi\right]\right]^0.
\end{equation*}
Using \eqref{sub-rep}, the right-hand side simplifies to
$$
\sum_{j=1}^{\f-1} \binom{\f\!-\!1}{j\!-\!1}^2 \left(\f\!-\!j\right)
c_j \, c_{\f-j}
\left[\T_{j-1}(g),\left[\T_{\f-j}(g),\varphi\right]\right]^0.
$$
In order to determine the variation of the term $\T_{\f-1}$, we
apply \eqref{var-t}. Thus we add
$$
-\frac{1}{n} c_\f \sum_{j=1}^{\f-1} j
\left[\T_{j-1}(g),\left[\T_{\f-j}(g),\varphi\right]\right]^0.
$$
Now a calculation shows that
$$
\binom{\f\!-\!1}{j\!-\!1}^2 \left(\f\!-\!j\right) c_j \, c_{\f-j} =
\frac{1}{n} c_\f j.
$$
The proof is complete.
\end{proof}

The infinitesimal conformal covariance \eqref{inf-pn} for all
metrics in the conformal class of $g$ implies the conformal
covariance
\begin{equation*}
e^{n\varphi} \PP_n (e^{2\varphi}g) =  \PP_n(g)
\end{equation*}
of
\begin{equation}\label{pp}
\PP_n \st \P_n^0 - c_\f \T_{\f-1}.
\end{equation}
In fact, we find
\begin{align*}
e^{n\varphi} \PP_n(e^{2\varphi}g) - \PP_n(g) & = \int_0^1 (d/dt)
\left( e^{n t\varphi} \PP_n(e^{2t\varphi} g) \right) dt \\
& = \int_0^1 e^{ns\varphi} (d/dt)|_0 \left(e^{nt\varphi}
\PP_n (e^{2t\varphi} (e^{2s\varphi} g))\right) ds \\
& = 0
\end{align*}
by the infinitesimal conformal covariance of $\PP_n$ for the metrics
$e^{2s\varphi}g$.

Thus, Theorem \ref{inv} enables us to derive (in the locally
conformally flat case) the conformal covariance of $\PP_n$ from the
presentations \eqref{sub-rep}, i.e.,
$$
P_{2N}^0 = \P_{2N}^0 - c_N \T_{N-1}
$$
for {\em all} subcritical GJMS-operators.

In the following section, we shall apply this argument to prove the
conformal covariance of $\PP_8$ (for locally conformally flat
metrics). The proof of Theorem \ref{P8-non-flat} extends the
argument to general metrics.

%%%%%%%%%%%%%%%%%%%%%%%%%%%%%%%%%%%%%%%%%%%%%%%%%%%%%%%%%%%%%%%%%%%%%%%%%%
\section{A conformally covariant fourth power of the Laplacian}\label{fourth}

As an application of Theorem \ref{inv} we have the following
construction of a conformally covariant fourth power of the
Laplacian (in the locally conformally flat category).

\begin{thm}\label{M8} In dimension $n=8$ and for locally
conformally flat metrics, the operator
\begin{equation}\label{pp8-flat}
\PP_8 = \P_8^0 - c_4 \delta( \Rho^3 \# d)
\end{equation}
with
$$
\P_8 = (3P_2P_6 + 3P_6P_2 + 9P_4^2) - (8P_2P_4P_2 + 12 P_2^2P_4 + 12
P_4P_2^2) + 18P_2^4
$$
is conformally covariant, i.e.,
$$
e^{8\varphi} \PP_8 (e^{2\varphi} g) = \PP_8(g)
$$
for all $\varphi \in C^\infty(M)$.
\end{thm}

\begin{proof} It is obvious that $\PP_8$ is of the form $\Delta^4 +
LOT$. By Theorem \ref{M4} and Theorem \ref{M6}, the operators
$\M_2^0$, $\M_4^0$ and $\M_6^0$ satisfy the relations
\eqref{sub-rep}. Theorem \ref{inv} implies the infinitesimal
conformal invariance of $\PP_8$.
\end{proof}

Conjecture \ref{formula} for $N=4$ extends Theorem \ref{M8} to
non-critical dimensions. It not only claims the conformal covariance
of the operator $\PP_8$, but also asserts that it coincides with
$P_8$. For the convenience of the reader, we restate that special
case in a form which also includes a description of the recursive
structure of the constant term $Q_8$ (for more details see Section
\ref{Q-rec}).

\begin{conj}\label{M8-gen} On locally conformally flat manifolds
of dimension $n \ge 3$, the GJMS-operator $P_8$ is given by
\begin{equation}\label{p8}
P_8 = \P_8^0 - 3!4!2^3 \delta( \Rho^3 \# d) + \left(\f\!-\!4\right)
Q_8,
\end{equation}
where
\begin{align*}
\P_8 & = (3 P_2 P_6 + 3 P_6 P_2 + 9 P_4^2) - (8 P_2 P_4 P_2 + 12
P^2_2 P_4 + 12 P_4 P_2^2) + 18 P_2^4 \\
& = \Delta^4 + LOT
\end{align*}
and
\begin{equation}\label{q8}
Q_8 = \Q_8 - 12 (Q_6 - \Q_6) Q_2 - 18 (Q_4 - \Q_4)^2 + 4!3!2^7 v_8,
\; v_8 = 2^{-4} \tr(\wedge^4 \Rho).
\end{equation}
The quantities $\Q_4$, $\Q_6$ and $\Q_8$ are displayed in Examples
\ref{qres-4}, \ref{qres-6} and \ref{qres-8}.
\end{conj}

\eqref{p6-general} can be derived from \eqref{q6} by conformal
variation. Therefore, it seems natural to prove that $\PP_8$
coincides with $P_8$ by conformal variation of \eqref{q8}. We will
return to this problem elsewhere. For an extension of Conjecture
\ref{M8-gen} to general metrics see Section \ref{n-flat}.

%%%%%%%%%%%%%%%%%%%%%%%%%%%%%%%%%%%%%%%%%%%%%%%%%%%%%%%%%%%%%%%%%%%%%%%%%%%%%%%%%
\section{Round spheres}\label{spheres}

\renewcommand{\thefootnote}{\arabic{footnote}}

On the round spheres $\S^n$, the GJMS-operators factor into
second-order operators (shifted Laplacians) according to the product
formula \eqref{product}. In particular, all GJMS-operators can be
written as universal polynomials in $P_2$:
\begin{equation}\label{univ}
P_{2N} = \prod_{j=0}^{N-1} (P_2 + j(j+1)), \; N \ge 1.
\end{equation}
Thus, \eqref{q-curv} yields

\begin{corr}\label{Q-S} On the round sphere $\S^n$,
\begin{equation}\label{Q-S-form}
Q_{2N} = \f \prod_{j=1}^{N-1} \left(\f-j\right)\left(\f+j\right), \;
N \ge 1.
\end{equation}
\end{corr}

The following refinement of Conjecture \ref{formula} also describes
the constant terms of the operators $\M_{2N}$.\footnote{The proofs
of Lemma \ref{develop}, Lemma \ref{ps} and Lemma \ref{finale} are
due to C. Krattenthaler.}

\begin{thm}\label{sphere} On the round spheres $\S^n$,
\begin{equation}\label{red}
\M_{2N} = N! (N\!-\!1)! P_2, \; N \ge 1.
\end{equation}
\end{thm}

For the proof of Theorem \ref{sphere} we split
$$
\M_{2N} = \sum_{|I|=N} m_I P_{2I}
$$
into the sum $\sum_{a=1}^N S_{(a,N)}$ of the partial sums
\begin{equation}\label{pa-su}
S_{(a,N)} \st \sum_{J, \, a + |J| = N} m_{(a,J)} P_{2a} P_{2J}.
\end{equation}

\begin{lemm}\label{develop} For all non-negative integers $A$ and $B$,
\begin{equation}\label{prep-op}
P_{2A} P_{2B} = \sum_{j=0}^A (-1)^j \frac{A ! B! (A\!+\!B)!}
{j!(A\!-\!j)!(B\!-\!j)!(A\!+\!B\!-\!j)!} P_{2(A+B-j)}.
\end{equation}
\end{lemm}

\begin{proof} \eqref{prep-op} is equivalent to the polynomial identity
\begin{equation}\label{prep}
p_{2A} p_{2B} = \sum_{j=0}^A (-1)^j \frac{A ! B! (A\!+\!B)!}
{j!(A\!-\!j)!(B\!-\!j)!(A\!+\!B\!-\!j)!} p_{2(A+B-j)},
\end{equation}
where
$$
p_{2N}(x) \st \prod_{j=0}^{N-1} (x+j(j+1)).
$$
We write
\begin{equation}\label{poch}
p_{2N}(-y(y+1)) = \prod_{j=0}^{N-1} (-y+j)(y+1+j) = (x_1)_N (x_2)_N
\end{equation}
with $x_1 = -y$ and $x_2 = y+1$. Now we substitute the (far)
right-hand side of \eqref{poch} on the right-hand side of
\eqref{prep} and write the result in hypergeometric notation:
\begin{multline*}
\sum_{j=0}^A (-1)^j \frac{A ! B!
(A\!+\!B)!}{j!(A\!-\!j)!(B\!-\!j)!(A\!+\!B\!-\!j)!}
(x_1)_{A+B-j} (x_2)_{A+B-j} \\
= (x_1)_{A+B} (x_2)_{A+B} \; {}_3 F_2 \left[-A\!-\!B,-A,-B ; \atop
1\!-\!A\!-\!B\!-\!x_1, 1\!-\!A\!-\!B\!-\!x_2 \right].
\end{multline*}
The ${}_3 F_2$-series can be evaluated by means of
Pfaff-Saalsch\"utz summation formula (\cite{AAR}, Theorem 2.2.6)
$$
{}_3 F_2 \left[a,b,-n ; \atop c, 1+a+b-c-n\right] =
\frac{(c-a)_n(c-b)_n}{(c)_n (c-a-b)_n},
$$
where $n$ is a non-negative integer. If we apply the formula (here
we use $x_1+x_2 = 1$), then after little simplification we obtain
$$
(x_1)_A (x_1)_B (x_2)_A (x_2)_B = (p_{2A} p_{2B})(-y(y\!+\!1)),
$$
i.e., the left-hand side of \eqref{prep}. \end{proof}

The following result provides a complete description of the partial
sums $S_{(a,N)}$. Its proof extends the proof of Lemma \ref{LPP}.

\begin{lemm}\label{ps} On $\S^n$,
\begin{multline}\label{ps-form}
\sum_{J, \; a+ |J|=N} m_{(a,J)} P_{2J} \\ = \binom{N\!-\!1}{a\!-\!1}
\sum_{k=0}^{N-a-1} (-1)^{N-a-k} \binom{N}{k}
\frac{(N\!-\!a)!(N\!-\!a\!-\!1)!}{(N\!-\!a\!-\!k)!
(N\!-\!a\!-\!k\!-\!1)!} P_{2(N-a-k)}
\end{multline}
for fixed $N \ge 2$ and $1 \le a \le N\!-\!1$.
\end{lemm}

Note that \eqref{strong} follows from \eqref{ps-form} by comparing
the coefficients of $\Delta^{|J|}$.

\begin{proof} We prove the claim by induction on $N$. Suppose that we have already
established \eqref{ps-form} up to $N-1$. By \eqref{bino}, the left-hand side of
\eqref{ps-form} equals
\begin{multline*}
- \frac{N!(N\!-\!1)!}{N \cdot a!(a\!-\!1)!(N\!-\!a)!(N\!-\!a\!-\!1)!} P_{2(N-a)} \\
- \frac{N!(N\!-\!1)!}{a!(a\!-\!1)!(N\!-\!a)!(N\!-\!a\!-\!1)!}
\sum_{b=1}^{N-a-1} \frac{1}{a+b} \sum_{K, \; b+|K|=N-a} m_{(b,K)} P_{2b} P_{2K}.
\end{multline*}
If we now use the induction hypothesis, then this sum simplifies to
\begin{multline*}
- \frac{N!(N\!-\!1)!}{N \cdot a!(a\!-\!1)!(N\!-\!a)!(N\!-\!a\!-\!1)!} P_{2(N-a)} \\
-\frac{N! (N\!-\!1)!}{a!(a\!-\!1)!(N\!-\!a)!(N\!-\!a\!-\!1)!} \sum_{b=1}^{N-a-1} \frac{1}{a+b}
\binom{N\!-\!a\!-\!1}{b\!-\!1} \\
\times \sum_{k=0}^{N-a-b-1} (-1)^{N-a-b-k} \binom{N\!-\!a}{k}
\frac{(N\!-\!a\!-\!b)! (N\!-\!a\!-\!b\!-\!1)!}
{(N\!-\!a\!-\!b\!-\!k)! (N\!-\!a\!-\!b\!-\!k\!-\!1)!}
P_{2(N-a-b-k)}P_{2b}.
\end{multline*}
The next step is to apply Lemma \ref{develop} to
$P_{2(N-a-b-k)}P_{2b}$. Thus, we arrive at the expression
\begin{multline*}
- \frac{N!(N\!-\!1)!}{N \cdot a!(a\!-\!1)!(N\!-\!a)!(N\!-\!a\!-\!1)!} P_{2(N-a)} \\
-\frac{N! (N\!-\!1)!}{a!(a\!-\!1)!} \sum_{b=1}^{N-a-1} \frac{1}{a+b}
\sum_{k=0}^{N-a-b-1} (-1)^{N-a-b-k}
\frac{(N\!-\!a\!-\!b\!-\!1)!}{k!(N\!-\!a\!-\!b\!-\!k\!-\!1)!} \\
\times \sum_{j=0}^{N-a-b-k} (-1)^j \frac{b}{j! (N\!-\!a\!-\!b\!-\!k\!-\!j)!(b\!-\!j)!
(N\!-\!a\!-\!k\!-\!j)!} P_{2(N-a-k-j)}.
\end{multline*}
At this point, we make an index transformation $s=j+k$. Then the
above expression can be written in the form
\begin{multline*}
- \frac{N!(N\!-\!1)!}{N \cdot a!(a\!-\!1)!(N\!-\!a)!(N\!-\!a\!-\!1)!} P_{2(N-a)} \\
- \frac{N! (N\!-\!1)!}{a!(a\!-\!1)!} \sum_{s=0}^{N-a-1} P_{2(N-a-s)} \sum_{b=1}^{N-a-1}
(-1)^{N-a-b-s} \frac{1}{a+b} \\ \times \frac{1}{(b-1)! (N\!-\!a\!-\!b\!-\!s)!(N\!-\!a\!-\!s)!}
\sum_{k=0}^s \binom{N\!-\!a\!-\!b\!-\!1}{k} \binom{b}{s\!-\!k}.
\end{multline*}
The sum over $k$ can be evaluated by means of the identity
$$
\sum_{k=0}^n \binom{r}{k} \binom{s}{n-k} = \binom{r+s}{n}
$$
(Vandermonde's convolution). Consequently, the above expression
simplifies to
\begin{multline*}
- \frac{N!(N\!-\!1)!}{N \cdot a!(a\!-\!1)!(N\!-\!a)!(N\!-\!a\!-\!1)!} P_{2(N-a)} \\
- \frac{N! (N\!-\!1)!}{a!(a\!-\!1)!} \sum_{s=0}^{N-a-1} P_{2(N-a-s)} \sum_{b=1}^{N-a-1}
(-1)^{N-a-b-s} \frac{1}{a\!+\!b} \\
\times \frac{1}{(b-1)! (N\!-\!a\!-\!b\!-\!s)!(N\!-\!a\!-\!s)!}
\binom{N\!-\!a\!-\!1}{s} \\
= - \frac{N! (N\!-\!1)!}{a!(a\!-\!1)!} \sum_{s=0}^{N-a-1} P_{2(N-a-s)} \sum_{b=1}^{N-a}
(-1)^{N-a-b-s} \frac{1}{a\!+\!b} \\
\times \frac{1}{(b-1)! (N\!-\!a\!-\!b\!-\!s)!(N\!-\!a\!-\!s)!}
\binom{N\!-\!a\!-\!1}{s}.
\end{multline*}
(The reader should observe the tiny difference in the summation
range for $b$ in the last line.) If we write the sum over $b$ in
hypergeometric notation, then we obtain the expression
\begin{multline*}
- \frac{N! (N\!-\!1)!}{a!(a\!-\!1)!} \sum_{s=0}^{N-a-1} P_{2(N-a-s)} \frac{1}{a\!+\!1} \\
\times \frac{1}{(N\!-\!a\!-\!s\!-\!1)!(N\!-\!a\!-\!s)!}
\binom{N\!-\!a\!-\!1}{s} {}_2 F_1 \left[a\!+\!1, a\!+\!s\!-\!N\!+\!1
; \atop a\!+\!2 \right].
\end{multline*}
The ${}_2 F_1$-series can be summed by means of the Chu--Vandermonde
summation formula (\cite{AAR}, Corollary 2.2.3)
$$
{}_2 F_1 \left[a,-n;\atop c\right] = \frac{(c-a)_n}{(c)_n},
$$
where $n$ is a non-negative integer. After some simplification, this
leads exactly to the right-hand side of \eqref{ps-form}.
\end{proof}

Lemma \ref{ps} shows that
\begin{equation}\label{s-partial}
\M_{2N} = (-1)^N \binom{N}{0} V_0^N + (-1)^{N-1} \binom{N}{1} V_1^N \pm
\cdots + \binom{N}{N\!-\!2} V_{N-2}^N,
\end{equation}
where
\begin{equation*}
V_0^N \st P_{2N} + \sum_{a=1}^{N-1} (-1)^a
\binom{N\!-\!1}{a\!-\!1} P_{2a} P_{2N-2a}
\end{equation*}
and
$$
V_k^N \st \sum_{a=1}^{N-1-k} (-1)^a \binom{N\!-\!1}{a\!-\!1}
\frac{(N\!-\!a)! (N\!-\!a\!-\!1)!}{(N\!-\!a\!-\!k)!
(N\!-\!a\!-\!k\!-\!1)!} P_{2a} P_{2N-2k-2a}
$$
for $k =1,\dots,N\!-\!2$. The following result proves Theorem
\ref{sphere}.

\begin{lemm}\label{finale} For all positive integers $N$, we have
\begin{equation}\label{final-sum}
\sum_{k=0}^{N-2} (-1)^{N-k} \binom{N}{k} V_k^N = N!(N\!-\!1)! P_2.
\end{equation}
\end{lemm}

\begin{proof} The assertion is equivalent to
\begin{multline}\label{fs}
P_{2N} + \sum_{k=0}^{N} (-1)^{N-k} \binom{N}{k} \\
\times \sum_{a=1}^{N-1-k}(-1)^a \binom{N\!-\!1}{a\!-\!1} \frac{(N\!-\!a)!(N\!-\!a\!-\!1)!}
{(N\!-\!a\!-\!k)! (N\!-\!a\!-\!k\!-\!1)!} P_{2a} P_{2N-2a-2k} = N! (N\!-\!1)! P_2
\end{multline}
(we apply the convention that empty sums vanish). We use Lemma
\ref{develop} to expand $P_{2a}P_{2N-2a-2k}$. Thus, the left-hand
side in \eqref{fs} becomes
\begin{multline*}
P_{2N} + \sum_{k=0}^{N} (-1)^{N-k} \binom{N}{k} \sum_{a=1}^{N-1-k} (-1)^a
\binom{N\!-\!1}{a\!-\!1} \frac{(N\!-\!a)!(N\!-\!a\!-\!1)!}{(N\!-\!a\!-\!k\!-\!1)!} \\
\times \sum_{j=0}^{a} (-1)^j \frac {a! (N\!-\!k)!} {j!(a\!-\!j)! (N\!-\!a\!-\!k\!-\!j)!
(N\!-\!k\!-\!j)!} P_{2(N-k-j)}.
\end{multline*}
We do again an index transformation: we let $s=k+j$ and hence
rewrite the above expression in the form
\begin{multline*}
P_{2N} + \sum_{s=0}^{N-1} \sum_{a=1}^{N-1} (-1)^{N+s+a} P_{2(N-s)}
\binom{N\!-\!1}{a\!-\!1} \frac {N!(N\!-\!a)!}{(N\!-\!a\!-\!s)!(N\!-\!s)!} \\
\times \sum_{k=0}^{N} \binom {N\!-\!a\!-\!1}k \binom{a}{s\!-\!k}.
\end{multline*}
The sum over $k$ can be evaluated by means of the Chu--Vandermonde
summation formula. Thus, we arrive at
\begin{multline*}
P_{2N} + \sum_{s=0}^{N-1} (-1)^{N+s}P_{2(N-s)} \frac
{N!(N\!-\!1)!^2}{(N\!-\!s)!s!(N\!-\!s\!-\!1)!^2}
\sum_{a=1}^{N-1} (-1)^a \binom{N\!-\!s\!-\!1}{a\!-\!1} \\
= \sum _{s=0}^{N-1} (-1)^{N+s} P_{2(N-s)} \frac {N!(N\!-\!1)!^2}{(N\!-\!s)!s!(N\!-\!s\!-\!1)!^2}
\sum_{a=1}^{N} (-1)^a \binom {N\!-\!s\!-\!1}{a\!-\!1}.
\end{multline*}
(The reader should observe the tiny difference in the summation
range for $a$ in the last line.) Finally, the binomial theorem
yields that the sum over $a$ always vanishes except if $s=N-1$. This
leads directly to the right-hand side of \eqref{fs}.
\end{proof}

Finally, we observe that for Einstein metrics with non-vanishing
scalar curvature, the formula in Conjecture \ref{formula} follows
from Theorem \ref{sphere} (and its analog on the real hyperbolic
space). In fact, for such metrics the GJMS-operators are given by
the product formula \eqref{einstein}, and a rescaling argument gives
\begin{equation}\label{rescal}
\M_{2N} = N!(N\!-\!1)! c^{N-1} P_2, \quad c =
\frac{\tau}{n(n\!-\!1)}.
\end{equation}
But this relation implies
$$
\M_{2N}^0 = - 2^{N-1} N!(N\!-\!1)! \delta (\Rho^{N-1}\#d)
$$
by using
$$
\Rho = \frac{\tau}{2n(n\!-\!1)}g.
$$

It is natural to summarize these results in terms of generating
functions. One should compare the following result with the version
\eqref{gf-cf-m} of Conjecture \ref{formula}.

\begin{corr}\label{gf-M-E} For Einstein metrics,
$$
\sum_{N \ge 1} \V_{2N} (r^2/4)^{N-1} = \delta \left((1-r^2/2 \,
\Rho)^{-2} \# d \right) + \left(\f-1\right) \tr (\Rho
(1-r^2/2\,\Rho)^{-2}),
$$
where $\V_{2N}$ is defined in \eqref{V-def}.
\end{corr}

\begin{proof} The identity $x (1-tx)^{-2} = \sum_{N \ge 1} N t^{N-1} x^N$ shows
that
\begin{equation*}
\sum_{N \ge 1} \V_{2N}^0 (r^2/4)^{N-1} = \sum_{N \ge 1} N \delta
(\Rho^{N-1} \#d) (r^2/2)^{N-1} = \delta ( (1-r^2/2 \, \Rho)^{-2} \#
d).
\end{equation*}
Moreover, \eqref{rescal} gives
\begin{align*}
\sum_{N \ge 1} \V_{2N}(1) (r^2/4)^{N-1} & = \left(\f\!-\!1\right)
\sum_{N \ge 1} N
\left( \frac{\tau}{n(n\!-\!1)}\right)^{N-1} \J (r^2/4)^{N-1} \\
& = \left(\f\!-\!1\right) \sum_{N \ge 1} N \left(
\frac{\tau}{2n(n\!-\!1)}\right)^N n (r^2/2)^{N-1} \\
& = \left(\f\!-\!1\right) \sum_{N \ge 1} N \tr (\Rho^N)
(r^2/2)^{N-1} \\
& = \left(\f\!-\!1\right) \tr (\Rho (1-r^2/2\,\Rho)^{-2}).
\end{align*}
The proof is complete.
\end{proof}

%%%%%%%%%%%%%%%%%%%%%%%%%%%%%%%%%%%%%%%%%%%%%%%%%%%%%%%%%%%%%%%%%%%%%%%%%%%%%
\section{Pseudo-spheres}\label{ps-sphere}

Here we discuss a special case of the literal extension of
Conjecture \ref{formula} to pseudo-Riemannian metrics. We consider
the conformally flat pseudo-spheres
$$
\S^{(q,p)} = \S^q \times \S^p, \; p \ge 1, \; q \ge 1
$$
with the metrics $g_{\S^q} - g_{\S^p}$ given by the round metrics on
the factors. Through this case, the theory is connected with
representation theory as follows. The Yamabe operators on the round
spheres have trivial kernels. But the kernel of the Yamabe operator
on $\S^{(q,p)}$ realizes an interesting infinite-dimensional
representation of $O(q+1,p+1)$. It was analyzed in detail in
\cite{ko-1} -- \cite{ko-3} and \cite{km}. These works illustrate the
interplay between conformal geometry, representation theory and
classical analysis. In particular, Kobayashi and {\O}rsted proved

\begin{thm}[\cite{ko-1}, Theorem 3.6.1] $\ker (P_2) \ne 0$ iff
$p+q \in 2 \n$. If $\ker(P_2) \ne 0$ and $(p,q) \ne (1,1)$, then the
kernel is an irreducible representation of $O(q\!+\!1,p\!+\!1)$ with
an unitarizable underlying Harish-Chandra module.
\end{thm}

More generally, all GJMS-operators are intertwining operators for
principal series representations of $O(q+1,p+1)$ which are induced
from a maximal parabolic subgroup. This fact leads to the following
reformulation of results of Mol{\v{c}}anov.

\begin{thm}[\cite{bran-2}, Theorem 6.2]\label{b-mol} On $\S^{(q,p)}$,
the GJMS-operators factorize as
\begin{align}\label{op-even}
P_{4N} & = \prod_{j=1}^N (B\!+\!C\!+\!(2j\!-\!1))(B\!-\!C\!-\!(2j\!-\!1))
(B\!+\!C\!-\!(2j\!-\!1))(B\!-\!C\!+\!(2j\!-\!1)) \nonumber \\
& = \prod_{j=1}^N \left[(B^2\!-\!C^2)^2\!-\!2(2j\!-\!1)^2
(B^2\!+\!C^2)\!+\!(2j\!-\!1)^4\right]
\end{align}
and
\begin{align}\label{op-odd}
P_{4N+2} & = (-B^2\!+\!C^2)
\prod_{j=1}^N (B\!+\!C\!+\!2j)(B\!-\!C\!-\!2j)(B\!+\!C\!-\!2j)(B\!-\!C\!+\!2j)
\nonumber \\
& = (-B^2\!+\!C^2) \prod_{j=1}^N \left[(B^2\!-\!C^2)^2\!-\!2(2j)^2
(B^2\!+\!C^2)\!+\!(2j)^4\right],
\end{align}
where
\begin{equation}\label{BC-pseudo}
B^2 = -\Delta_{\S^q} + \left(\frac{q-1}{2}\right)^2 \quad \mbox{and}
\quad C^2 = -\Delta_{\S^p} + \left(\frac{p-1}{2}\right)^2.
\end{equation}
\end{thm}

\begin{corr}\label{Q-PS} On $\S^{(q,p)}$,
\begin{equation}\label{Q-PS-form}
Q_{2N} = \prod_{j=1}^{N-1} \left( \frac{p+q}{2}\!+\!N\!-\!2j \right)
\prod_{j=0}^{N-1} \left( \frac{q-p}{2}\!-\!N\!+\!1\!+\!2j \right),
\; N \ge 1.
\end{equation}
\end{corr}

For $p=0$, we have $C^2 = 1/4$ and a calculation shows that the
product formulas \eqref{op-even} and \eqref{op-odd} specialize to
\eqref{product}. Moreover, \eqref{Q-PS-form} is easily seen to
specialize to \eqref{Q-S-form}.

The following result extends Theorem \ref{sphere}.

\begin{thm}[\cite{jk}]\label{pseudo} On $\S^{(q,p)}$,
\begin{equation}\label{pseudo-even}
\M_{4N} = (2N)!(2N\!-\!1)! \left(\frac{1}{2}\!-\!B^2\!-\!C^2\right),
\; N \ge 1
\end{equation}
and
\begin{equation}\label{pseudo-odd}
\M_{4N+2} = (2N\!+\!1)!(2N)! (-B^2\!+\!C^2), \; N \ge 0.
\end{equation}
\end{thm}

The proof of Theorem \ref{pseudo} rests on an extension of Lemma
\ref{ps}.

In view of $P_2 = -B^2+C^2$, the identities \eqref{pseudo-even} and
\eqref{pseudo-odd} are equivalent to the non-linear relations
\begin{equation}\label{nonlinear}
2 \M_{4N} = (2N)!(2N\!-\!1)! \left(P_4 - P_2^2\right) \quad
\mbox{and} \quad \M_{4N+2} = (2N\!+\!1)!(2N)! P_2
\end{equation}
of intertwining operators for $O(q+1,p+1)$.

Now Theorem \ref{pseudo} implies
\begin{align*}
\M_{4N}^0 & = (2N)!(2N\!-\!1)! (\Delta_{\S^q} + \Delta_{\S^p}), \; N \ge 1, \\
\M_{4N+2}^0 & = (2N\!+\!1)!(2N)! (\Delta_{\S^q} - \Delta_{\S^p}), \;
N \ge 0.
\end{align*}
But using
$$
2 \Rho = \begin{pmatrix} 1_q & 0 \\ 0 & -1_p
\end{pmatrix},
$$
these identities can be written in the form
\begin{align*}
\M_{4N}^0 & = - c_{2N} \delta (\Rho^{2N-1} \# d), \\
\M_{4N+2}^0 & = - c_{2N+1} \delta(\Rho^{2N} \# d).
\end{align*}
In other words, Theorem \ref{pseudo} confirms a special case of the
literal extension of Conjecture \ref{formula} to pseudo-Riemannian
metrics.

Finally, we observe that an easy calculation using the relations
\begin{align*}
(q/2-1) q/2 - (p/2-1) p/2 & = (q-1)^2/4
- (p-1)^2/4, \\
(q/2-1) q/2 + (p/2-1) p/2 & = (q-1)^2/4 + (p-1)^2/4 - 1/2
\end{align*}
yields the following analog of Corollary \ref{gf-M-E}.

\begin{corr}\label{gf-M-PS} On $\S^{q,p}$,
\begin{multline*}
\sum_{N \ge 1} \V_{2N} \, (r^2/4)^{N-1} \\[-4mm] = \delta((1 - r^2/2 \,
\Rho)^{-2} \#d) + \tr \left( \begin{pmatrix} q/2-1 & 0 \\ 0 &
p/2-1\end{pmatrix} \Rho (1-r^2/2 \, \Rho)^{-2}\right).
\end{multline*}
\end{corr}

%%%%%%%%%%%%%%%%%%%%%%%%%%%%%%%%%%%%%%%%%%%%%%%%%%%%%%%%%%%%%%%%%%%%%%%%%%%%
\section{Residue polynomials}\label{polynomial}

We use the GJMS-operators $P_{2N}$ on $M$ to define a sequence of
polynomial families $P_{2N}^{res}(\lambda)$, $N \ge 1$ of
differential operators on $M$. Their definition is motivated by the
definition of the $Q$-polynomials
\begin{equation}\label{Q-pol-def}
Q_{2N}^{res}(\lambda) \st (-1)^{N-1} D_{2N}^{res}(\lambda)(1)
\end{equation}
as the constant terms of the residue families
$D_{2N}^{res}(\lambda)$ of \cite{J-book}. We describe the relation
to the operators $\M_{2N}$.

\begin{defn}\label{def-e} Let $P_2^{res}(\lambda) = P_2$, and
define the families $P_{2N}^{res}(\lambda)$ for $N \ge 2$
recursively by
\begin{multline}\label{rec-def}
P_{2N}^{res}(\lambda) = \prod_{k=1}^{N-1}
\left( \frac{\lambda+\f-2N+k}{k} \right) P_{2N} \\
+ \sum_{j=1}^{N-1} (-1)^j \prod_{k=1 \atop k \ne j}^N \left(
\frac{\lambda+\f-2N+k}{k-j} \right) P_{2j}
P_{2N-2j}^{res}\left(-\f\!+\!2N\!-\!j\right).
\end{multline}
\end{defn}

Formula \eqref{rec-def} can be regarded as a Lagrange interpolation
formula.

\begin{thm}\label{interpol} $P_{2N}^{res}(\lambda)$ is the unique
polynomial of degree $N-1$ which is characterized by the conditions
\begin{equation*}
P_{2N}^{res} \left(-\f+N \right) = (-1)^{N-1} P_{2N}
\end{equation*}
and
$$
P_{2N}^{res} \left(-\f\!+\!2N\!-\!j\right) = (-1)^j P_{2j}
P_{2N-2j}^{res} \left(-\f\!+\!2N\!-\!j\right)
$$
for $j=1,\dots,N-1$.
\end{thm}

By resolving the recursions, we find
\begin{equation}\label{rep-form}
P_{2N}^{res}(\lambda) = \sum_{|I|=N} a_I(\lambda) P_{2I}
\end{equation}
with polynomial coefficients $a_I(\lambda)$ of degree $N\!-\!1$.

\begin{conj}\label{mystic}
\begin{equation}\label{main}
\M_{2N} = (d/d \lambda)^{N-1}|_0 (P_{2N}^{res}(\lambda)), \; N \ge
2.
\end{equation}
\end{conj}

The relation \eqref{main} can be confirmed by computer calculations
for not too large $N$.

Thus, under Conjecture \ref{mystic},
\begin{equation*}
P_{2N}^{res}(\lambda) = \CC_{2N}^{N-1}
\frac{(\lambda\!+\!\f\!-\!N)^{N-1}}{(N\!-\!1)!} + \CC_{2N}^{N-2}
\frac{(\lambda\!+\!\f\!-\!N)^{N-2}}{(N\!-\!2)!} + \cdots +
\CC_{2N}^0
\end{equation*}
with
$$
\CC_{2N}^{N-1} = \M_{2N}.
$$
The coefficients are given by universal, i.e., dimension
independent, linear combinations of compositions of GJMS-operators.
In particular, $\CC_{2N}^0 = (-1)^{N-1} P_{2N}$ and the critical
polynomial $P_{n}^{res}(\lambda)$ has the form
\begin{equation}\label{crit-p-pol}
P^{res}_{n}(\lambda) = \M_{n} \frac{\lambda^{\f-1}}{(\f\!-\!1)!} +
\cdots + (-1)^{\f-1} P_{n}.
\end{equation}
However, only the leading coefficient and the constant term of
$P_{2N}^{res}(\lambda)$ are self-adjoint.

Thus, under Conjecture \ref{mystic}, Conjecture \ref{formula}
describes the leading coefficient $\CC^{N-1}_{2N}$ of the polynomial
$P_{2N}^{res}(\lambda)$ as a certain differential operator of {\em
second} order. More generally, we expect that the coefficients
$\CC_{2N}^j$ are differential operators of respective orders
$2N\!-\!2j$. Identifying these operators yields additional recursive
formulas.

The operator $\P_{2N}$ describes the majority of contributions to
$P_{2N}$. It is accompanied by a scalar curvature quantity which
plays a similar role in recursive formulas for the $Q$-curvature
$Q_{2N}$ (see Section \ref{Q-rec}). We recall that
$$
\P_{2N} = - \sum_{|I|=N, \, I \ne (N)} m_I P_{2I}.
$$

\begin{defn}\label{Q-leading} For $N \ge 2$, we set
\begin{equation}\label{main-Q}
(-1)^N \Q_{2N} = - \sum_{a+|J|=N, \, a \ne N} m_{(J,a)} (-1)^a
P_{2J}(Q_{2a}).
\end{equation}
\end{defn}

The first few of these curvature quantities read as follows.

\begin{ex}\label{qres-4} $\Q_4 = -P_2(Q_2)$.
\end{ex}

\begin{ex}\label{qres-6} $\Q_6 = -2 P_2(Q_4) + 2P_4(Q_2) - 3 P_2^2(Q_2)$.
\end{ex}

\begin{ex}\label{qres-8}
\begin{equation*}
\Q_8 = -3 P_2(Q_6) - 3 P_6(Q_2) + 9 P_4(Q_4) + 8 P_2 P_4(Q_2)  - 12
P_2^2 (Q_4) + 12 P_4 P_2 (Q_2) - 18 P_2^3(Q_2).
\end{equation*}
\end{ex}

The conjectural appearance of $\P_{2N}$ in the leading coefficient
of $P_{2N}^{res}(\lambda)$ has an analog for $\Q_{2N}$: it is
conjectured to appear in the leading coefficient $\L_{2N}$ of the
polynomial $Q_{2N}^{res}(\lambda)$.

In the subcritical case $2N < n$, the polynomial
$Q_{2N}^{res}(\lambda)$ is a polynomial of degree $N$ which is
recursively determined by the $N$ relations
\begin{equation}\label{character-A}
Q_{2N}^{res}\left(-\f\!+\!2N\!-\!j\right) = (-1)^j P_{2j}
Q_{2N-2j}^{res}\left(-\f\!+\!2N\!-\!j\right)
\end{equation}
for $j=1,\dots,N-1$ and
\begin{equation}\label{character-B}
Q_{2N}^{res}\left(-\f\!+\!N\right) = -\left(\f\!-\!N\right) Q_{2N},
\end{equation}
together with either
\begin{equation}\label{character-C}
Q_{2N}^{res}(0) = 0
\end{equation}
or a formula which relates $\dot{Q}_{2N}^{res}(-\f\!+\!N)$ and
$Q_{2N}$. The critical $Q$-polynomial $Q_n^{res}(\lambda)$ is
recursively determined by the relations
\begin{equation*}
Q_n^{res}\left(\f\!-\!j\right) = (-1)^j P_{2j}
Q_{n-2j}^{res}\left(\f\!-\!j\right)
\end{equation*}
for $j=1,\dots,\f-1$,
\begin{equation*}
Q_n^{res}(0) = 0,
\end{equation*}
and the identity
$$
\dot{Q}_n^{res}(0) = Q_n.
$$
In full generality, these characterizations are conjectural. For
background and full details on residue families we refer to
\cite{J-book}.

In these terms, we conjecture that
\begin{multline}\label{Taylor-q-res}
Q_{2N}^{res}(\lambda) = \L_{2N} \left(\lambda\!+\!\f\!-\!N\right)^N
+ \cdots \\[-3mm] = (-1)^N \left[\Q_{2N} - Q_{2N}\right]
\frac{(\lambda\!+\!\f\!-\!N)^N}{(N\!-\!1)!} + \cdots.
\end{multline}
Similarly, the critical polynomial $Q_n^{res}(\lambda)$ is
conjectured to have the form
\begin{equation}\label{crit-q-pol}
Q_n^{res}(\lambda) = (-1)^\f \left[\Q_n - Q_n\right]
\frac{\lambda^\f}{(\f\!-\!1)!} + \cdots + Q_n \lambda.
\end{equation}
The fact that the critical $Q$-curvature $Q_n$ appears as the linear
term is a consequence of the holographic formula \cite{gj}.

We describe the role of Conjecture \ref{mystic} and of the related
conjectural relation \eqref{Taylor-q-res} between $Q_{2N} - \Q_{2N}$
and $\L_{2N}$. This relation is the source of a conjectural
recursive description of $Q$-curvatures in terms of lower order
$Q$-curvatures and the volume of Poincar\'e-Einstein metrics. This
will be discussed in Section \ref{Q-rec}. By conformal variation,
the resulting formulas for $Q$-curvatures imply formulas for the
corresponding GJMS-operators which naturally contain the primary
parts $\P_{2N}$.

We illustrate the idea by considering the special case of $Q_6$ in
the critical dimension $n=6.$ $Q_6^{res}(\lambda)$ is a polynomial
of degree $3$. Its characterizing properties imply that it has the
form
$$
Q_6^{res} (\lambda) = \frac{\lambda^3}{2!} (Q_6 + 2P_2(Q_4) -
2P_4(Q_2) + 3P_2^2(Q_2)) + \cdots,
$$
i.e.,
$$
Q_6^{res} (\lambda) = \frac{\lambda^3}{2!} (Q_6 - \Q_6) + \cdots
$$
(see Example \ref{qres-6}). On the other hand, an evaluation of the
definition of $Q_6^{res}(\lambda)$ as the constant term of
$D_6^{res}(\lambda)$ (see the discussion in \cite{J-book}, Section
6.11) shows that the quantity
$$
\Lambda_6 = Q_6 - \Q_6
$$
coincides with
$$
-6 (Q_4 + P_2(Q_2)) \cdot Q_2 - 2!3!2^5 v_6,
$$
where $v_6$ is defined by \eqref{hol-coeff}. The resulting identity
is a recursive formula for $Q_6$ in terms of $P_4$, $P_2$, $Q_4$,
$Q_2$ {\em and} $v_6$. Now recall that conformal variation of $Q_6$
yields the non-constant part $P_6^0$ of $P_6$. By the above
representation of $Q_6$, the resulting formula for $P_6$ reads
$$
P_6 = (2P_2 P_4 + 2P_4 P_2 - 3 P_2^3) + \cdots = \P_6 + \cdots.
$$
In fact, Theorem \ref{M6} shows that $\P_6$ covers {\em all} but a
certain second-order term.

The latter result suggests to generate similar recursive formulas
for $P_{2N}$ by conformal variation of the leading coefficient of
the polynomial $Q_{2N}^{res}(\lambda)$. Through Definition
\ref{Q-leading}, Conjecture \ref{mystic} connects the recursively
defined leading coefficient of $Q_{2N}^{res}(\lambda)$ with the
definition of $\M_{2N}$ in Definition \ref{coefficient}. Theorem
\ref{c-cv} shows that the conformal variation of $\P_{2N}$ is only a
second-order operator.

%%%%%%%%%%%%%%%%%%%%%%%%%%%%%%%%%%%%%%%%%%%%%%%%%%%%%%%%%%%%%%%%%%%%%%
\section{Universal recursive formulas for $Q$-curvatures}\label{Q-rec}

In the present section, we discuss universal recursive formulas for
$Q$-curvatures. They describe $Q_{2N}$ as the sum of its {\em
primary part} $\Q_{2N}$ and its {\em secondary part}. The primary
part is determined by lower order $Q$-curvatures and lower order
GJMS-operators. We discuss two equivalent descriptions of the
secondary parts, and confirm the general picture for round spheres
and pseudo-spheres. A different type of recursive formulas for
$Q$-curvatures was discussed in \cite{FJ} and \cite{J-book}.

It is natural to compare the formula in Conjecture \ref{Q-quadrat}
with the holographic formula of \cite{gj} which relates the critical
$Q$-curvature $Q_n$ of a Riemannian manifold $(M,g)$ of even
dimension $n$ to the holographic coefficients $v_2, \dots, v_n$. We
start by recalling this identity. $(M,g)$ gives rise to a
Poincar\'e-Einstein metric
$$
g_+ = r^{-2}(dr^2 + g_r), \; g_0=g
$$
on the space $(0, \varepsilon) \times M$ (for sufficiently small
$\varepsilon$). The coefficients in the formal Taylor series
\begin{equation}\label{hol-coeff}
v(r) = \frac{vol(g_r)}{vol(g)} = 1+ v_2 r^2 + v_4 r^4 + \cdots + v_n
r^n + \cdots
\end{equation}
are functionals of the metric $g$. These are the renormalized volume
coefficients of \cite{G-vol}, \cite{G-ext} (called holographic
coefficients in \cite{J-book}). For locally conformally flat
metrics, the functionals $v_{2j}$ are given by the formula
\begin{equation}\label{holo-coeff}
v_{2j} = (-1)^j \frac{1}{2^j} \tr (\wedge^j \Rho).
\end{equation}
The functionals \cite{V}
$$
\sigma_j = \tr (\wedge^j \Rho)
$$
give rise to the so-called $\sigma_j$-Yamabe problem which,
in recent years, has been studied intensively. However, in
dimensions $n \ge 6$, these studies are restricted to the locally
conformally flat case \cite{bg-var}. It was suggested in \cite{CF}
that for general metrics the functionals $v_{2j}$ should be regarded
as natural substitutes.

Now let $u$ be an eigenfunction of the Laplacian of $g_+$, i.e.,
$-\Delta_{g_+} u = \lambda(n-\lambda)u$. Its formal asymptotics
$$
u \sim \sum_{j \ge 0} r^{\lambda+2j} \T_{2j}(\lambda)(f), \; r \to 0
$$
defines a sequence of rational families of differential operators
$\T_{2j}(\lambda)$ on $C^\infty(M)$. These should not be confused
with the operators $\T_j$ in \eqref{second}. Let
$\T_{2j}^*(\lambda)$ denote the formal-adjoint operator with respect
to the metric $g$. Then \cite{gj}
\begin{equation}\label{hol}
(-1)^\f \left(\left(\f\right)!\left(\f\!-\!1\right)! \,
2^{n-1}\right)^{-1} Q_n = \frac{1}{n} \left(\sum_{j=1}^{\f-1} (n-2j)
\T_{2j}^*(0)(v_{n-2j}) \right) + v_n.
\end{equation}
This formulation uses the conventions of \cite{J-book}. For a
discussion of an extension of \eqref{hol} to subcritical
$Q$-curvatures we refer to \cite{J-book}, Section 6.9.

As described in Section \ref{polynomial}, the recursive formulas for
GJMS-operators in Conjecture \ref{formula} are suggested by
recursive formulas for $Q$-curvatures which arise from its relation
to the $Q$-polynomials. For $Q_4$ and $Q_6$, these recursive
formulas read
\begin{equation}\label{rec-q-4}
(Q_4 - \Q_4) + Q_2^2 = 2! 2^3 v_4
\end{equation}
and
\begin{equation}\label{rec-q-6}
(Q_6 - \Q_6) + 6(Q_4-\Q_4)Q_2 = -2! 3! 2^5 v_6
\end{equation}
for
$$
\Q_4 = - P_2(Q_2) \quad \mbox{and} \quad \Q_6 = -2P_2(Q_4) + 2
P_4(Q_2) - 3 P_2^2(Q_2)
$$
(see Example \ref{qres-4} and Example \ref{qres-6}). Next, in
\cite{J-book}, Section 6.13 we derived the formula
\begin{equation}\label{rec-q-8}
(Q_8 - \Q_8) + 12 (Q_6-\Q_6)Q_2 + 18 (Q_4-\Q_4)^2 = 3!4!2^7 v_8
\end{equation}
with $\Q_8$ as in Example \ref{qres-8} (for $n=8$ and under some
technical assumption which probably can be removed).

In order to formulate an extension of \eqref{rec-q-4} --
\eqref{rec-q-8}, it will be convenient to use the notation
\begin{equation}\label{defect}
\Lambda_{2j} \st Q_{2j} - \Q_{2j}.
\end{equation}

\begin{conj}[{\bf Universal recursive formulas for $Q$-curvatures}]
\label{Q-quadrat} For even $n$ and $2N \le n$,
\begin{equation}\label{Q-gen-rec}
2 \Lambda_{2N} + \sum_{j=1}^{N-1} \frac{j(N\!-\!j)}{N}
\binom{N}{j}^2 \Lambda_{2N-2j} \Lambda_{2j} = (-1)^N N!(N\!-\!1)!
2^{2N} v_{2N}.
\end{equation}
For odd $n$, the same relations hold true for all $N$.
\end{conj}

It should be emphasized that the recursive relations in Conjecture
\ref{Q-quadrat} are much simpler than those discussed in \cite{FJ}.

\begin{ex} For $N=3$ and $N=4$, \eqref{Q-gen-rec} reads
$$
2 \Lambda_6 + 6 \Lambda_2 \Lambda_4 + 6 \Lambda_4 \Lambda_2 = -2!3!
2^6 v_6
$$
and
$$
2 \Lambda_8 + 12 \Lambda_2 \Lambda_6 + 36 \Lambda_4^2 + 12 \Lambda_6
\Lambda_2 = 3!4! 2^8 v_8.
$$
These identities are equivalent to \eqref{rec-q-6} and
\eqref{rec-q-8}, respectively.
\end{ex}

It is a common feature of \eqref{hol} and \eqref{Q-gen-rec} that
they describe the differences
$$
Q_{2N} - (-1)^N N!(N\!-\!1)! 2^{2N-1} v_{2N}.
$$
But both descriptions do this in fundamentally different ways. We
compare \eqref{Q-gen-rec} for the critical $Q$-curvature with the
holographic formula \eqref{hol}. Both formulas provide expressions
for the difference
$$
Q_n - (-1)^\f  \left(\f\right)!\left(\f\!-\!1\right)! 2^{n-1} v_n.
$$
The following observation points to the subtleties of the relation.
For the critical $Q_6$ we compare \eqref{rec-q-6}, i.e.,
\begin{equation}\label{ex-rec-6}
Q_6 = -2P_2(Q_4) + (2P_4 - 3P_2^2)(Q_2) - 6(Q_4 - \Q_4)Q_2 - 2!3!2^5
v_6
\end{equation}
with the holographic formula
\begin{equation}\label{hol-6}
Q_6 = -2^6 \left(4 \T_2^*(0)(v_4) + 2 \T_4^*(0)(v_2)\right) -
2!3!2^5 v_6.
\end{equation}
The operators $\T_2(0)$ and $\T_4(0)$ have the form
$$
\T_2(0) = 2^{-3}(\Delta + \cdots), \quad \T_4(0) = 2^{-6} (\Delta^2
+ \cdots).
$$
The term $\Delta^2 \J = \Delta^2 Q_2$ contributes to $Q_6$ with the
coefficient $1$. In \eqref{hol-6}, this terms is captured by
$\T_4^*(0)(v_2)$ using $v_2 = -\frac{1}{2} \J$. On the other hand,
both terms $P_2(Q_4)$ and $(2P_4 - 3P_2^2)(Q_2)$ in \eqref{ex-rec-6}
are required to cover this contribution.

Note that \eqref{Q-gen-rec} is compatible with the result
\cite{bran-2} that, in $Q_{2N}$, the contribution with the largest
number of derivatives is $(-1)^{N-1} \Delta^{N-1}(\J)$. In fact, the
primary part $\Q_{2N}$ of $Q_{2N}$ (see Definition \ref{Q-leading})
contains the contribution
\begin{align*}
(-1)^N \sum_{a+|J|=N, \, a \ne N} m_{(J,a)} \Delta^{|J|}
\Delta^{a-1}(\J) & = (-1)^N \left( \sum_{|I|=N, \; I \ne (N)} m_I
\right) \Delta^{N-1}(\J) \\
& = (-1)^{N-1} \Delta^{N-1}(\J) \qquad \mbox{(by \eqref{sum-0})}.
\end{align*}

A repeated application of \eqref{Q-gen-rec} yields a formula for
$Q_{2N}$ as a sum of the primary part $\Q_{2N}$ and a linear
combination of terms of the form
$$
v_{2I} = v_{2I_1} \cdots v_{2I_r}
$$
for $I=(I_1,\dots,I_r)$ with $|I| = N$. Before we describe the
structure of these formulas, we display the first few special cases.

\begin{ex}\label{v-formula} The identity $Q_2 = -2 v_2$ is trivial.
Moreover, we have
$$
Q_4 = \Q_4 + 4 (4 v_4 - v_2^2)
$$
and
$$
Q_6 = \Q_6 - 48 (8 v_6 - 4 v_4 v_2 + v_2^3).
$$
\end{ex}

Now let
\begin{equation}\label{G-def}
\G(r) \st 1 + \sum_{N \ge 1} (-1)^N \Lambda_{2N}
\frac{r^N}{N!(N\!-\!1)!}.
\end{equation}

Ignoring the problems caused by obstructions, the following
conjecture reformulates Conjecture \ref{Q-quadrat} in terms of
generating functions.

\begin{conj}[{\bf Duality}]\label{Q-G}
\begin{equation}\label{effective}
\G\left(\frac{r^2}{4}\right) = \sqrt{v(r)},
\end{equation}
where $v$ is defined by \eqref{hol-coeff}.
\end{conj}

In fact, \eqref{Q-gen-rec} is equivalent to
$$
2 \frac{\Lambda_{2N}}{N!(N\!-\!1)!} + \sum_{j=1}^{N-1}
\frac{\Lambda_{2N-2j}}{(N\!-\!j)!(N\!-\!1\!-\!j)!}
\frac{\Lambda_{2j}}{j!(j\!-\!1)!} = (-1)^N 2^{2N} v_{2N}.
$$
This yields the equivalence.

The formulation of Conjecture \ref{Q-G} can be taken literally, for
instance, for locally conformally flat metrics. For general metrics
and even $n$, the possibly existing obstructions require to
interpret \eqref{effective} as an identity of terminating Taylor
series.

By definition, the generating function $\G$ lives on $M$ and its
variable $r$ is a formal variable. \eqref{effective} relates it to
the volume form of an associated Poincar\'e-Einstein metric on a
space $X =(0,\varepsilon) \times M$ of one more dimension. In this
connection, the variable $r$ has a geometric meaning. The relation
\eqref{effective} resembles the statements around the
AdS/CFT-duality which, for instance, claim relations between
super-string theory on $AdS_5$ and super-Yang-Mills theory on its
boundary \cite{W}. The common flavor motivates to refer to
\eqref{effective} as a duality. Its validity in the special case of
$X=AdS_5$ with the boundary $M=\S^{3,1}$ is one of the facts which
support the general formulation (see Corollary \ref{Q-c-pseudo} and
the remarks following it).

A formal calculation of the square root in \eqref{effective} yields
a power series in even powers of $r$ with coefficients that are
linear combinations of terms of the form $v_{2I}$. More precisely,
\begin{equation}
\sqrt{v(r)} = 1 + w_2 r^2 + w_4 r^4 + w_6 r^6 + \cdots,
\end{equation}
where
\begin{align*}
2 w_2 & = v_2, \\
2 w_4 & = \frac{1}{4} \left( 4 v_4 - v_2^2 \right), \\
2 w_6 & = \frac{1}{8} \left( 8 v_6 - 4 v_4 v_2 + v_2^3 \right).
\end{align*}

Now comparing coefficients in \eqref{effective}, yields
\begin{equation}\label{Beta}
(-1)^N (Q_{2N} - \Q_{2N}) = 2^{2N} N!(N\!-\!1)! w_{2N}.
\end{equation}
The first three of these identities are given in Example
\ref{v-formula}. The next relation
\begin{equation}\label{next}
Q_8 - \Q_8 = 2^8 4! 3! w_8
\end{equation}
with
$$
2 w_8 = \frac{1}{64}\left( 64 v_8 - 32 v_6 v_2 - 16 v_4^2 + 24 v_2^2
v_4 - 5 v_2^4 \right)
$$
is already implicit in the proof of Theorem 6.13.1 in \cite{J-book}
(without the identification of $w_8$, however).

The relations \eqref{Beta} replace the description \eqref{Q-gen-rec}
of the secondary parts $Q_{2N} - \Q_{2N}$ by a description in terms
of holographic coefficients.

Next, we prove Conjecture \ref{Q-quadrat} and Conjecture \ref{Q-G}
for the round spheres. The proofs rest on the following result.

\begin{lemm}\label{Q-sphere} On $\S^n$, we have
\begin{equation}\label{lambda-form}
\frac{\Lambda_{2N}}{(N\!-\!1)!} = \prod_{j=0}^{N-1}
\left(\f-j\right), \; N \ge 1.
\end{equation}
Hence
$$
\G(r) = \sum_{N \ge 0} (-1)^N \binom{\f}{N} r^N = (1-r)^\f.
$$
\end{lemm}

Now for $\S^n$, we have \cite{G-vol}
$$
g_r = (1 - r^2/4)^2 g \quad \mbox{and} \quad v(r) = (1-r^2/4)^n.
$$
This proves \eqref{effective}. In other words, we have proved

\begin{thm}\label{Q-red-sphere} Conjecture \ref{Q-quadrat} and
Conjecture \ref{Q-G} hold true on round spheres.
\end{thm}

We continue with the proof of Lemma \ref{Q-sphere}.

\begin{proof} By Definition \ref{Q-leading} and \eqref{defect}, the assertion
is equivalent to
\begin{equation}\label{Q-claim}
\sum_{a+|J|=N} m_{(J,a)} (-1)^a P_{2J}(Q_{2a}) =
(-1)^N (N\!-\!1)! \prod_{j=0}^{N-1} \left(\f-j\right).
\end{equation}
Corollary \ref{self} shows that, on general manifolds, the left-hand
side of \eqref{Q-claim} equals
\begin{align*}
\sum_{a=1}^N (-1)^a \left( \sum_{|J|=N-a} m_{(J,a)} P_{2J} \right) (Q_{2a})
& = \sum_{a=1}^N (-1)^a \left( \sum_{|J|=N-a} m_{(a,J)} P_{2 J^{-1}} \right) (Q_{2a}) \\
& = \sum_{a=1}^N (-1)^a  \left( \sum_{|J|=N-a} m_{(a,J)} P_{2J} \right)^* (Q_{2a}).
\end{align*}
Now by Lemma \ref{ps}, the latter sum simplifies to
\begin{multline}\label{Q-pf-2}
(-1)^N Q_{2N} + \sum_{a=1}^{N-1} (-1)^a \binom{N\!-\!1}{a\!-\!1} \\
\times \sum_{k=0}^{N-a-1} (-1)^{N-a-k} \binom{N}{k}
\frac{(N\!-\!a)!(N\!-\!a\!-\!1)!}{(N\!-\!a\!-\!k)!(N\!-\!a\!-\!k\!-\!1)!}
P_{2N-2a-2k}(Q_{2a}).
\end{multline}
But since $Q$-curvatures of round spheres are constant,
$$
P_{2N-2a-2k}(Q_{2a}) = (-1)^{N-a-k} \left(\f\!-\!\left(N\!-\!a\!-\!k\right)\right)
Q_{2N-2a-2k} Q_{2a}.
$$
Hence \eqref{Q-pf-2} can be written as
\begin{multline}\label{Q-pf-3}
(-1)^N Q_{2N} + \sum_{a=1}^{N-1} (-1)^a \binom{N\!-\!1}{a\!-\!1} \\
\times \sum_{k=0}^{N-a-1} \binom{N}{k} \frac{(N\!-\!a)!
(N\!-\!a\!-\!1)!}{(N\!-\!a\!-\!k)!(N\!-\!a\!-\!k\!-\!1)!}
\left(\f\!-\!\left(N\!-\!a\!-\!k\right)\right) Q_{2N-2a-2k} Q_{2a}.
\end{multline}
Now we express the products $Q_{2N-2a-2k} Q_{2a}$ as linear
combinations of $Q's$. Taking constant terms in Lemma \ref{develop},
yields
\begin{multline*}
\left(\f\!-\!a\right)\left(\f\!-\!\left(N\!-\!a\!-\!k\right)\right) Q_{2a} Q_{2N-2a-2k} \\
= \sum_{j=0}^a \frac{a!(N\!-\!a\!-\!k)!(N\!-\!k)!}{j!(a\!-\!j)!(N\!-\!a\!-\!k\!-\!j)!(N\!-\!k\!-\!j)!}
\left(\f\!-\!\left(N\!-\!k\!-\!j\right)\right)Q_{2N-2k-2j}.
\end{multline*}
By similar arguments as in the proof of Lemma \ref{finale},
\eqref{Q-pf-3} leads to
\begin{multline*}
(-1)^N Q_{2N} + \sum_{s=0}^{N-1} \sum_{a=1}^{N-1} (-1)^a \frac{(\f\!-\!(N\!-\!s))}{(\f\!-\!a)}
Q_{2N-2s} \binom{N\!-\!1}{a\!-\!1} \frac{N!(N\!-\!a)!}{(N\!-\!a\!-\!s)!(N\!-\!s)!} \\
\times \sum_{k=0}^N \binom{N\!-\!a\!-\!1}{k} \binom{a}{s\!-\!k}.
\end{multline*}
As in the proof of Lemma \ref{finale}, the sum over $k$ equals
$\binom{N-1}{s}$, and we get
\begin{equation}\label{Q-pf-4}
\sum_{s=0}^{N-1} \left(\f\!-\!\left(N\!-\!s\right)\right) Q_{2N-2s}
\frac{N!(N\!-\!1)!^2}{(N\!-\!s)! s! (N\!-\!s\!-\!1)!^2} \sum_{a=1}^N
(-1)^a \frac{1}{(\f-a)} \binom{N\!-\!s\!-\!1}{a\!-\!1}.
\end{equation}
Now the well-known partial fraction expansion
$$
\sum_{k = 0}^n (-1)^k \binom{n}{k} \frac{1}{x+k} =
\frac{n!}{x(x+1)\cdots(x+n)}, \; x \ne 0,-1-,\dots,-n
$$
(see \cite{cm}, p. 188) simplifies \eqref{Q-pf-4} to
\begin{multline}\label{intermed}
N!(N\!-\!1)!^2 \\ \times \sum_{s=0}^{N-1}
\left(\f\!-\!\left(N\!-\!s\right)\right) Q_{2N-2s}
\frac{1}{(-\f\!+\!1) \cdots (-\f\!+\!N\!-\!s)} \frac{1}{(N\!-\!s)!
\, s! \, (N\!-\!s\!-\!1)!},
\end{multline}
i.e.,
\begin{equation}\label{Q-pf-5}
N!(N\!-\!1)!^2 (-1)^N \sum_{s=0}^{N-1} (-1)^s \frac{\f (\f\!+\!1)
\cdots (\f\!+\!(N\!-\!s\!-\!1))}{(N\!-\!s)! \, s! \,
(N\!-\!s\!-\!1)!}.
\end{equation}
In the last step, we used the product formula (see Corollary
\ref{Q-S})
\begin{equation}\label{QN-sphere}
Q_{2N} = \f \prod_{j=1}^{N-1}
\left(\f\!-\!j\right)\left(\f\!+\!j\right).
\end{equation}
Now we write \eqref{Q-pf-5} in the form
\begin{equation*}
- N!(N\!-\!1)!^2 \sum_{s=0}^{N-1} (-1)^s \frac{\f (\f\!+\!1) \cdots
(\f\!+\!s)}{(s\!+\!1)! \, (N\!-\!1\!-\!s)!\, s!} = -N! (N\!-\!1)!
\sum_{s=0}^{N-1} (-1)^s \binom{N-1}{s} \binom{\f+s}{\f-1}
\end{equation*}
and apply the identity (\cite{cm}, (5.24))
$$
\sum_k (-1)^k \binom{l}{m+k} \binom{s+k}{n} = (-1)^{l+m}
\binom{s-m}{n-l}, \quad l,m,n \in \Z, \; l \ge 0
$$
(which can be proved by induction on $l$). Thus, we find
$$
N! (N\!-\!1)! (-1)^N \binom{\f}{N} = (-1)^N (N\!-\!1)!
\prod_{j=0}^{N-1} \left(\f-j\right).
$$
This proves \eqref{Q-claim}. In the above arguments, we have
suppressed the following subtlety in the critical case $2N=n$. The
sum \eqref{Q-pf-4} involves an undefined term for $s=0$ and $a=\f$.
For these parameters, the fraction $(\f-(N-s))/(\f-a)$ is to be
interpreted as $1$. Likewise, in \eqref{intermed} the undefined
fraction for $s=0$ is to be interpreted appropriately. The proof is
complete. \end{proof}

Lemma \ref{Q-sphere} supports the conjectural relation
\eqref{Taylor-q-res}
$$
\L_{2N} = - (-1)^N \frac{\Lambda_{2N}}{(N-1)!}
$$
between the secondary part $Q_{2N}-\Q_{2N}$ and the leading
coefficient $\L_{2N}$ of $Q_{2N}^{res}(\lambda)$. In fact, for the
sphere $\S^n$, the following result describes the full polynomials
$Q_{2N}^{res}(\lambda)$.

\begin{lemm} On $\S^n$,
$$
Q_{2N}^{res}(\lambda) = -(-1)^N \f \prod_{j=1}^{N-1}
\left(\f-j\right) \lambda \prod_{j=1}^{N-1} (\lambda-N-j)
$$
for all $N \ge 1$.
\end{lemm}

\begin{proof} In the non-critical case $2N \ne n$, it suffices to
verify that the given expression satisfies the characterizing
properties
$$
Q_{2N}^{res}\left(-\f\!+\!2N\!-\!j\right) = (-1)^j P_{2j}
Q_{2N-2j}^{res}\left(-\f\!+\!2N\!-\!j\right), \; j=1,\dots,N-1
$$
and
\begin{equation}\label{second-set}
Q_{2N}^{res}\left(-\f\!+\!N\right) = -\left(\f\!-\!N\right) Q_{2N},
\quad Q_{2N}^{res}(0)=0
\end{equation}
(see \eqref{character-A}, \eqref{character-B} and
\eqref{character-C}). In the critical case $2N=n$, the two
conditions in \eqref{second-set} coincide, and have to be
supplemented by $\dot{Q}_n^{res}(0) = Q_n$. The proof is
straightforward and we omit the details.
\end{proof}

Finally, we outline a proof of Conjecture \ref{Q-G} for the
conformally flat pseudo-spheres of Section \ref{ps-sphere}. Full
details are given in \cite{jk}. First, we make the assertion more
explicit. By conformal flatness, the Poincar\'e-Einstein metric is
given by $g_+ = r^{-2} (dr^2 + g_r)$ with $g_r = g - \Rho r^2 +
\Rho^2 r^4/4$ (\cite{FG2}, \cite{J-book}). Now using
\begin{equation}\label{pseudo-Schouten}
\Rho = \frac{1}{2} \begin{pmatrix} g_{\S^q} & 0 \\ 0 & g_{\S^p}
\end{pmatrix} \quad \mbox{and} \quad \Rho^2 = \frac{1}{4} \begin{pmatrix}
g_{\S^q} & 0 \\ 0 & -g_{\S^p}
\end{pmatrix},
\end{equation}
we find
\begin{equation} \label{pseudo-PM}
g_+ = r^{-2}(dr^2 + (1-r^2/4)^2 g_{\S^q} - (1+r^2/4)^2 g_{\S^p}).
\end{equation}
In particular, the volume function $v(r)$ is given by
\begin{equation}\label{pseudo-volume}
v(r) = (1-r^2/4)^q (1+r^2/4)^p.
\end{equation}
It follows that Conjecture \ref{Q-G} is equivalent to
$$
1 + \sum_{N \ge 1} \frac{r^N}{N!(N\!-\!1)!} \left( \sum_{a+|J|=N}
(-1)^a m_{(J,a)} P_{2J} (Q_{2a}) \right) = (1-r)^{\frac{q}{2}}
(1+r)^{\frac{p}{2}}.
$$
But since $Q$-curvatures are constant, this, in turn, is equivalent
to the summation formulas
\begin{equation}\label{pseudo-final}
\sum_{|I|=N} m_I \frac{P_{2I}(1)}{\f-I_{\text{last}}} = N!
(N\!-\!1)! \left(\sum_{M=0}^N (-1)^M \binom{\frac{q}{2}}{M}
\binom{\frac{p}{2}}{N\!-\!M} \right), \; N \ge 1,
\end{equation}
where $I_{\text{last}}$ denotes the last entry of the composition
$I$.

\begin{thm}[\cite{jk}]\label{Q-sum-pseudo} \eqref{pseudo-final} holds true.
\end{thm}

The proof of Theorem \ref{Q-sum-pseudo} rests on an extension of
Lemma \ref{ps}.

\begin{corr}\label{Q-c-pseudo} Conjecture \ref{Q-G} holds true on
pseudo-spheres.
\end{corr}

Note that
$$
g_+ = r^{-2}(dr^2 + (1-r^2/4)^2 g_{\S^{n-1}} - (1+r^2/4)^2 g_{\S^1})
$$
is nothing else than the metric on the anti-de Sitter space
$AdS_{n+1}$ of dimension $n+1$. In fact, $AdS_{n+1}$ is defined as
the hyper-surface
$$
C= \left\{ x_1^2+\cdots+x_{n-1}^2-y_1^2-y_2^2=-1\right\} \subset
\r^{n+2}
$$
with the metric induced by $g_0 = dx_1^2 + \cdots + dx_{n-1}^2 -
dy_1^2 - dy_2^2$. The map
$$
\kappa: (0,1) \times \S^{n-1} \times \S^1 \ni (r,x,y) \mapsto \left(
\frac{1\!-\!r^2}{2r} x,\frac{1\!+\!r^2}{2r} y\right) \in C
$$
pulls back $g_0$ to
$$
\frac{1}{r^2}\left(dr^2 + \left(1-r^2\right)^2 \frac{1}{4}
g_{\S^{n-1}} - \left(1+r^2\right)^2 \frac{1}{4} g_{\S^1} \right),
$$
and the substitution $r \mapsto r/2$ yields $g_+$.

Thus, a special case of Corollary \ref{Q-c-pseudo} is a duality
which relates the volume function $v$ of anti-de Sitter space
$AdS_{n+1}$ to the generating function $\G$ on its boundary
$\S^{n-1,1}$.

%%%%%%%%%%%%%%%%%%%%%%%%%%%%%%%%%%%%%%%%%%%%%%%%%%%%%%%%%%%%%%%%%%%%%%%%%%
\section{A related family of examples}\label{mix}

In the present section, we confirm the general picture for a family
of Riemannian metrics with terminating Poincar\'e-Einstein metrics
discussed in \cite{GL}. The results basically follow from the
corresponding results for pseudo-spheres $\S^{q,p}$.

We consider product manifolds $M^n = \S^q \times \h^p$, $n=q+p$ with
product metrics $g_{\S^q} + g_{\h^p}$ given by the respective
constant curvature metrics of curvature $\pm 1$ on the factors. The
following result describes the associated Poincar\'e-Einstein
metrics.

\begin{thm}[\cite{GL}, Theorem 4.1]\label{GL-case} The metric
\begin{equation}\label{mixed-PM}
g_+ = r^{-2}\left(dr^2 + \left(1-r^2/4\right)^2 g_{\S^q} +
\left(1+r^2/4\right)^2 g_{\h^p}\right)
\end{equation}
on $(0,2) \times M^n$ satisfies
$$
\Ric(g_+) = -n g_+.
$$
\end{thm}

The corresponding result in \cite{GL} is actually more general: the
factors $\S^q$ and $\h^p$ can be replaced by Einstein spaces with
suitably related scalar curvatures. The main feature of these
metrics is that their Schouten tensors decompose as the sum of the
respective Schouten tensors of the factors. For $\S^q \times \h^p$,
\begin{equation}\label{mixed-Schouten}
\Rho_{\S^q \times \h^p} =
\begin{pmatrix} \Rho_{\S^q} & 0 \\ 0 & \Rho_{\h^p} \end{pmatrix}
= \frac{1}{2}
\begin{pmatrix} g_{\S^q} & 0 \\ 0 & - g_{\h^p} \end{pmatrix}.
\end{equation}
The metric \eqref{mixed-PM} and the decomposition
\eqref{mixed-Schouten} should be compared with the metric
\eqref{pseudo-PM} and the decomposition \eqref{pseudo-Schouten}.
Note that $\tr(\Rho) = (q-p)/2$ in both cases.

The GJMS-operators $P_{2N}$ on $M^n$ can be obtained by replacing
$\Delta_{\S^p}$ by $-\Delta_{\h^p}$ in Theorem \ref{b-mol}. In fact,
by \cite{GZ} the GJMS-operators appear in the asymptotics of the
eigenfunctions of the Laplacian of the corresponding
Poincar\'e-Einstein metric. Hence the explicit formulas
\eqref{pseudo-PM} and \eqref{mixed-PM} imply the claim. Thus, we
have

\begin{thm}\label{mixed} On $\S^q \times \h^p$, the GJMS-operators
factorize as
\begin{equation}\label{op-e}
P_{4N} = \prod_{j=1}^N \left[(B^2\!-\!C^2)^2\!-\!2(2j\!-\!1)^2
(B^2\!+\!C^2)\!+\!(2j\!-\!1)^4\right]
\end{equation}
and
\begin{equation}\label{op-o}
P_{4N+2} = (-B^2\!+\!C^2) \prod_{j=1}^N
\left[(B^2\!-\!C^2)^2\!-\!2(2j)^2 (B^2\!+\!C^2)\!+\!(2j)^4\right],
\end{equation}
where
\begin{equation}\label{BC-mixed}
B^2 = -\Delta_{\S^q} + \left(\frac{q-1}{2}\right)^2 \quad \mbox{and}
\quad C^2 = \Delta_{\h^p} + \left(\frac{p-1}{2}\right)^2.
\end{equation}
\end{thm}

In particular,
$$
P_2=-B^2+C^2 = \Delta_{\S^q} + \Delta_{\h^p} - \left(\f-1\right) Q_2
$$
with
$$
Q_2 = \tr (\Rho) = \frac{q-p}{2}
$$
and
$$
P_4 = (\Delta_{\S^q} + \Delta_{\h^p})^2 - \left(\f\!-\!1\right)
(q\!-\!p) (\Delta_{\S^q} \!+\! \Delta_{\h^p}) + 2 (\Delta_{\S^q}
\!-\!\Delta_{\h^p}) + \left(\f\!-\!2\right)Q_4
$$
with
\begin{equation*}\label{mixed-q4}
Q_4 = \frac{(p+q)}{2} \frac{(q-p-2)}{2} \frac{(q-p+2)}{2}.
\end{equation*}

The operators $P_{2N}$, acting on functions which are constant on
one of the factors, can be written as products of shifted Laplace
operators. More precisely, we identify the restriction of $P_{2N}$
to functions which are constant on $\h^p$ with an operator
$P_{2N}^+$ on $C^\infty(\S^q)$. Similarly, we identify the
restriction of $P_{2N}$ to functions which are constant on $\S^p$
with an operator $P_{2N}^-$ on $C^\infty(\h^p)$. The following
product formulas generalize the product formula \eqref{product} on
$\S^n$ and its analog \eqref{einstein} (for $\tau = -n(n-1)$) on
$\h^n$.

\begin{corr}\label{mixed-rest}
\begin{equation}\label{p-plus}
P_{2N}^+ = \prod_{j=0}^{N-1} \left(\Delta_{\S^q} +
\left(\frac{p+q}{2}-N+2j\right)\left(\frac{p-q}{2}-N+2j+1\right)\right)
\end{equation} and
\begin{equation}\label{p-minus}
P_{2N}^- = \prod_{j=0}^{N-1} \left(\Delta_{\h^p} +
\left(\frac{p+q}{2}-N+2j\right)\left(\frac{p-q}{2}+N-2j-1\right)\right).
\end{equation}
\end{corr}

\begin{proof} For the operators $P_{4N}^+$, we observe that the product
of the factors for $j=k$ and $2N\!-\!1\!-\!k$ in \eqref{p-plus}
coincides with the factor for $j=N-k$ in \eqref{op-e}. In
\eqref{p-plus} for $P_{4N+2}^+$, the factor for $j=N$ is $P_2^+$.
The remaining $2N$ factors group together similarly as before.
Analogous arguments apply for the operators $P_{2N}^-$.
\end{proof}

Corollary \ref{mixed-rest} has some consequences for the kernel of
the critical GJMS-operator $P_n$ on $\S^q \times \h^p$, $q+p=n$. By
the conformal covariance of $P_n$, this space is an invariant of the
conformal class of the metric. On functions which are constant on
$\h^p$, $P_n$ reduces to
$$
\prod_{j=0}^{\f-1} (\Delta_{\S^q} + 2j(2j+1-q)).
$$
The latter formula shows that the $SO(q+1)$-modules
$$
\ker (\Delta_{\S^q} + L(L+q-1)) = \ker (\Delta_{\S^q} + 2j(2j+1-q)),
$$
where
$$
L=q\!-\!1\!-\!2k \quad \mbox{and} \quad j=q\!-\!1\!-\!k \quad
\mbox{with} \quad k=0,\cdots,\lfloor (q\!-\!1)/2 \rfloor,
$$
of spherical harmonics of degree $L$ induce subspaces of $\ker
(P_n)$. This generalizes an observation of M. Eastwood and M. Singer
for $P_4$ on $\S^2 \times \h^2$. Similarly, on functions which are
constant on $\S^q$, $P_n$ reduces to
$$
\prod_{j=0}^{\f-1} (\Delta_{\h^p} + 2j(p-1-2j)).
$$
Hence the non-trivial $SO(1,p)$-modules
$$
\ker (\Delta_{\h^p} + 2j(p-1-2j)), \; j=0,\dots,\lfloor
(p-1)/2\rfloor,
$$
induce subspaces of $\ker (P_n)$.

These observations extend to the critical GJMS-operator $P_n$ on
compact spaces of the form $\S^q \times \Gamma \backslash \h^p$,
where the discrete subgroup $\Gamma \subset SO(1,p)$ operates with a
compact quotient on $\h^p$.

The meaning of a non-trivial kernel of $P_n$ for the $Q$-curvature
prescription problem on compact manifolds was discussed in
\cite{G-kernel}.

The following result should be compared with Corollary \ref{Q-PS}.

\begin{corr}\label{Q-mixed} On $\S^q \times \h^p$,
$$
Q_{2N} = \prod_{j=1}^{N-1} \left( \frac{p+q}{2}\!+\!N\!-\!2j \right)
\prod_{j=0}^{N-1} \left( \frac{q-p}{2}\!-\!N\!+\!1\!+\!2j \right),
\; N \ge 1.
$$
\end{corr}

Corollary \ref{Q-mixed} implies that the critical $Q$-curvature
$Q_n$, $n=q+p$, of $\S^q \times \h^p$ equals
$$
Q_n = \prod_{j=1}^{\f-1} (n-2j) \prod_{j=0}^{\f-1} (2j+1-p).
$$
This yields

\begin{corr}\label{vanishing-Q} For odd $q$ and $p$, the critical $Q$-curvature
of $\S^q \times \h^p$ vanishes.
\end{corr}

Alternatively, this result follows from the holographic formula
\cite{gj}. In fact, by Remark 6.16.1 in \cite{J-book}, $Q_n$ is a
constant multiple of the coefficient of $r^n$ in
$$
\left(1-\frac{r^2}{2}\right)^q \left(1+\frac{r^2}{2}\right)^p.
$$
Hence it suffices to prove that
$$
\sum_{i=0}^q (-1)^i \binom{q}{i} \binom{n-q}{\f-i} = 0
$$
for odd $q$. But this sum equals
$$
\frac{1}{2} \left( \sum_{i=0}^q (-1)^i \binom{q}{i}
\binom{n-q}{\f-i} + (-1)^q \sum_{i=0}^q (-1)^i \binom{q}{q-i}
\binom{n-q}{\f-(q-i)} \right).
$$
Hence it vanishes for odd $q$.

Corollary \ref{vanishing-Q} also appears in \cite{chen}.

Now direct calculations yield the recursive formulas
$$
P_4 = (P_2^2)^0 + 2 (\Delta_{\S^q} - \Delta_{\h^p})+
\left(\f-2\right) Q_4
$$
and
$$
P_6 = \left(2P_2 P_4 + 2P_4 P_2 - 3 P_2^3\right)^0 + 12 (
\Delta_{\S^q} + \Delta_{\h^p}) - \left(\f-3\right) Q_6
$$
with
\begin{equation*}
Q_6 = \frac{(p+q-2)}{2}\frac{(p+q+2)}{2}
\frac{(q-p-4)}{2}\frac{(q-p)}{2}\frac{(q-p+4)}{2}.
\end{equation*}
These are special cases of Theorem \ref{M4} and Theorem \ref{M6},
respectively.

More generally, the following summation formula is a Riemannian
analog of Theorem \ref{pseudo}.

\begin{thm}\label{sum-mixed} On $\S^q \times \h^p$,
\begin{equation*}
\M_{4N} = (2N)! (2N\!-\!1)! \left( \frac{1}{2}\!-\!B^2\!-\!
C^2\right), \; N \ge 1
\end{equation*}
and
\begin{equation*}
\M_{4N+2} = (2N\!+\!1)! (2N)! \, P_2, \; N \ge 0
\end{equation*}
with $B^2$ and $C^2$ as defined in \eqref{BC-mixed}.
\end{thm}

In particular, we find
$$
\M_{2N}^0 = - 2^{N-1} N! (N\!-\!1)! \, \delta (\Rho^{N-1} \# d)
$$
using \eqref{mixed-Schouten}. This result is a special case of
Conjecture \ref{formula}. The volume function $v$ of $g_+$ is given
by
$$
v(r) = \left(1-r^2/4\right)^q \left(1+r^2/4\right)^p
$$
and the relation
$$
\G(r^2/4)= \sqrt{v(r)}
$$
follows from the corresponding relation for the pseudo-spheres.

%%%%%%%%%%%%%%%%%%%%%%%%%%%%%%%%%%%%%%%%%%%%%%%%%%%%%%%%%%%%%%%%%%%%%%%%%%
\section{Extension beyond conformally flat metrics}\label{n-flat}

\renewcommand{\thefootnote}{\arabic{footnote}}

In the present section, we formulate an extension of Conjecture
\ref{formula} to general metrics, and discuss that extension for
$N=4$ in the critical dimension $n=8$. In this case, the formulation
involves the first two of Graham's extended obstruction tensors
\cite{G-ext}.

In order to motivate the following formulations, we note that
\eqref{p6-general} for general metrics differs from the formula for
locally conformally flat metrics only by the second-order operator
\begin{equation}\label{add-6}
-\frac{16}{n-4} \delta(\B \# d).
\end{equation}
Moreover, \eqref{q6} and \eqref{v6} show that $P_6$, when viewed as
a rational function in $n$, has a simple pole at $n=4$ with residue
$$
\R_6 = -16( \delta(\B \#d) - (\B,\Rho)).
$$
A direct proof shows that $\R_6$ is conformally covariant, i.e.,
$$
e^{5\varphi} \circ \hat{\R}_6 = \R_6 \circ e^{-\varphi}.
$$
Here $\hat{\R}_6$ denotes $\R_6$ for the metric $\hat{g} =
e^{2\varphi} g$. The same convention will be used in the following
for other tensors. The operator $\R_6$ obstructs the existence of
$P_6$ for metrics with $\C \ne 0$ in dimension $n=4$.

We interpret \eqref{add-6} as
\begin{equation}\label{omega-1}
2! 2^3 \delta (\Omega^{(1)} \#d)
\end{equation}
using the first extended obstruction tensor
\begin{equation}\label{ext-first}
\Omega^{(1)} = \frac{\B}{4-n}.
\end{equation}
Similarly, for general metrics, the formula for $\PP_8$ in
Conjecture \ref{M8-gen} should contain the additional term
\begin{equation}\label{obstruct-8}
3! 2^4 \delta(\Omega^{(2)} \# d),
\end{equation}
where $\Omega^{(2)}$ is the second extended obstruction tensor
\cite{G-ext}.

The first two extended obstruction tensors are defined by
\begin{equation}\label{omega-1-2}
\Omega^{(1)}_{ij} = \tilde{R}_{\infty i j \infty}|_{\rho=0,t=1}
\quad \mbox{and} \quad \Omega^{(2)}_{ij} = \tilde{\nabla}_\infty
(\tilde{R})_{\infty i j \infty}|_{\rho=0,t=1},
\end{equation}
where $\tilde{R}$ denotes the curvature tensor of the
Fefferman-Graham ambient metric. The extended obstruction tensors
$\Omega^{(k)}$ are special conformal curvature tensors (in the sense
of \cite{FG2}). In particular, they vanish if $\C = 0$, and their
conformal variations only depend on first-order derivatives of
$\varphi$.

The tensors $\Omega^{(k)}$ can be regarded as rational functions in
the dimension $n$. The Schouten tensor $\Rho$ and the extended
obstruction tensors play the role of canonical building blocks of
the ambient metric and hence also of derived quantities such as the
holographic coefficients. For the details we refer to \cite{G-ext}.

The following result proves the conformal covariance of a
generalization of $\PP_8$ (Theorem \ref{pp8-flat}) in the critical
dimension.

\begin{thm}\label{P8-non-flat} On manifolds of dimension $n=8$,
the self-adjoint operator
\begin{equation}\label{P8-general}
\PP_8 \st \P_8^0 - 3!2^4 \delta(\left[ \Omega^{(2)} - 4 (\Rho
\Omega^{(1)} + \Omega^{(1)} \Rho) + 12 \Rho^3 \right] \# d)
\end{equation}
is conformally covariant, i.e.,
\begin{equation}
e^{8 \varphi} \PP_8 (e^{2 \varphi} g)  = \PP_8(g)
\end{equation}
for all metrics $g$ and all $\varphi \in C^\infty(M)$.
\end{thm}

Although the definition of $\PP_8$ involves the first two extended
obstruction tensors, the following proof does not depend on the
explicit formulas \cite{G-ext} for these tensors in terms of $\Rho$,
$\C$, $\Y$ and $\B$. It is natural to ask whether $\PP_8$ coincides
with $P_8$. As already noted in Section \ref{fourth}, it seems
natural to approach this problem by conformal variation of $Q_8$ on
the basis of the recursive formula \eqref{rec-q-8} and the explicit
expression \eqref{v8} for $v_8$ derived in \cite{G-ext}.

\begin{proof} It suffices to prove the infinitesimal conformal covariance.
By Theorem \ref{c-cv},
\begin{equation*}
(d/dt)|_0 \left(e^{8t\varphi} \P^0_8(e^{2t\varphi}g) \right) = 9
[\M_6,[\M_2,\varphi]]^0 + 18 [\M_4,[\M_4,\varphi]]^0 + 3
[\M_2,[\M_6,\varphi]]^0.
\end{equation*}
Since $\M_2$, $\M_4$ and $\M_6$ are second-order operators,
$$
(d/dt)|_0 \left(e^{8t\varphi} \P^0_8(e^{2t\varphi}g) \right) = 9
[\M^0_6,[\M^0_2,\varphi]]^0 + 18 [\M^0_4,[\M^0_4,\varphi]]^0 + 3
[\M^0_2,[\M^0_6,\varphi]]^0.
$$
The relations
$$
\M_6^0 = - 48 \T_2 - \frac{16}{n\!-\!4} \delta(\B \#d), \quad \M_4^0
= -4 \T_1 \quad \mbox{and} \quad \M_2^0 = \Delta
$$
imply that the above sum coincides with
\begin{multline}\label{reduced}
144 \left( -3 [\T_2,[\Delta,\varphi]]^0 + 2 [\T_1,[\T_1,\varphi]]^0
- [\Delta,[\T_2,\varphi]]^0 \right) \\
-\frac{16}{n\!-\!4} \left( 9 [\delta(\B \# d),[\Delta,\varphi]]^0 +
3 [\Delta,[\delta(\B \# d),\varphi]]^0 \right).
\end{multline}
By the same arguments as in the conformally flat case, it only
remains to prove that the sum of
\begin{equation}\label{BB}
144 \left( 8 \Y_{rij} \Rho_k^r + 4 \Y_{krj} \Rho_i^r + 4 \Y_{kir}
\Rho_j^r \right) \varphi^i  \Hess^{jk}(u)
\end{equation}
(see \eqref{non-flat}) and the main part of
\begin{equation}\label{AA}
-\frac{16}{n\!-\!4} \left( 9 [\delta(\B \# d),[\Delta,\varphi]]^0 +
3 [\Delta,[\delta(\B \# d),\varphi]]^0 \right)
\end{equation}
coincides with
$$
96 \times \mbox{the main part of the conformal variation of
$\delta([\Omega^{(2)} - 4 (\Rho \Omega^{(1)} + \Omega^{(1)} \Rho)]
\# d)$}.
$$
Now \eqref{AA} has the main part
\begin{equation}\label{8-main}
-\frac{16}{n\!-\!4} \left[ ( 18 \nabla_i(\B)_{jk} - 12
\nabla_k(\B)_{ij} ) \varphi^i - 48 \B_{ij} \Hess_k^i(\varphi)
\right] \Hess^{jk}(u).
\end{equation}
By the transformation law of the Bach tensor (see
\eqref{bach-conform}),
$$
e^{2\varphi} \hat{\Omega}^{(1)}_{ij} = \Omega^{(1)}_{ij} - (\Y_{kij}
+ \Y_{kji}) \varphi^k - \C_{kijl} \varphi^k \varphi^l.
$$
Hence the main part of the conformal variation of $\frac{1}{2}
\delta( (\Rho \Omega^{(1)} + \Omega^{(1)} \Rho) \# d)$ is
\begin{equation}
\left[ \Omega^{(1)}_{ir} \Hess_j^r(\varphi) + (\Y_{lir} + \Y_{lri})
\Rho_j^r \varphi^l \right] \Hess^{ij}(u).
\end{equation}
In order to determine the conformal variation of $\Omega^{(2)}$, we
apply some results of \cite{G-ext}. Under conformal changes,
$$
e^{4\varphi} \hat{\Omega}^{(2)}_{ij} = \Omega^{(2)}_{ij} -
\Y^{(2)}_{ijl} \varphi^l + O(|\nabla \varphi|^2)
$$
(Proposition 2.7) with the second Cotton tensor $\Y^{(2)}$
(Definition 2.4).\footnote{We retain the convention of \cite{G-ext}
concerning the higher Cotton tensor: $\Y^{(2)}$ is symmetric in the
first two arguments. On the other hand, $\Y$ is anti-symmetric in
the first two arguments.} Now the relation
$$
\Y^{(2)}_{ijk} = \left( 3 \tilde{\nabla}_k(\tilde{R})_{\infty ij \infty} -
\tilde{\nabla}_j(\tilde{R})_{\infty li \infty} -
\tilde{\nabla}_i(\tilde{R})_{\infty l j \infty} \right) \big|_{\rho=0,t=1}
$$
and the formula
$$
\tilde{\nabla}_l (\tilde{R})_{\infty ij \infty}|_{\rho=0,t=1} =
\frac{\nabla_l(\B)_{ij}}{4-n} - (\Y_{rij} + \Y_{rji}) \Rho_l^r
$$
for the covariant derivatives of $\tilde{R}$ suffice to determine
the main part of the conformal variation of $\delta(\Omega^{(2)} \# d)$.
It is given by the sum of
\begin{multline}\label{8-A}
\frac{1}{4-n} \left( 3 \nabla_l(\B)_{ij} - \nabla_j(\B)_{li} -
\nabla_i(\B)_{lj} \right) \varphi^l \Hess^{ij}(u) \\
= \frac{1}{4-n} \left( 3 \nabla_l(\B)_{ij} - 2 \nabla_j(\B)_{li}
\right) \varphi^l \Hess^{ij}(u)
\end{multline}
and
\begin{equation*}
\left( - 3(\Y_{rij}+ \Y_{rji}) \Rho^r_l +
2 (\Y_{ril} + \Y_{rli}) \Rho_j^r \right) \varphi^l \Hess^{ij}(u).
\end{equation*}
Now $96 \times \mbox{\eqref{8-A}}$ coincides with the terms in
\eqref{8-main} which contain a derivative of $\B$. Next, the
relation
$$
96 \cdot(-8) \Omega_{ir}^{(1)} \Hess_j^r(\varphi) \Hess^{ij}(u) =
\frac{16}{n-4} 48 \B_{ij} \Hess_k^i(\varphi) \Hess^{jk}(u)
$$
verifies the assertion for the contributions which contain two
derivatives of $\varphi$. It only remains to prove that
\begin{multline}\label{8-C}
144 \left(8 \Y_{rij} \Rho_k^r + 4 \Y_{krj} \Rho_i^r
+ 4 \Y_{kir} \Rho_j^r \right) \varphi^i \Hess^{jk}(u) \\
=  144 \left(8 \Y_{rli} \Rho_j^r + 4 \Y_{jri} \Rho_l^r
+ 4 \Y_{jlr} \Rho_i^r \right) \varphi^l \Hess^{ij}(u)
\end{multline}
coincides with
\begin{equation}\label{8-D}
 96 \left[-6 \Y_{rji} \Rho_l^r + 2(\Y_{ril} + \Y_{rli}) \Rho_j^r
- 8 (\Y_{lir} + \Y_{lri}) \Rho^r_j \right] \varphi^l \Hess^{ij}(u).
\end{equation}
We observe that $144 \cdot 4 \Y_{jri} = - 96 \cdot 6 \Y_{rji}$.
Next, for the last two products in \eqref{8-D} we find
\begin{multline*}
\big[ 2 \nabla_r(\Rho)_{il} - 2 \nabla_i(\Rho)_{rl} + 2 \nabla_r(\Rho)_{li}
- 2 \nabla_l(\Rho)_{ri} \\
- 8 \nabla_l(\Rho)_{ir} + 8 \nabla_i(\Rho)_{lr} - 8 \nabla_l(\Rho)_{ri}
+ 8 \nabla_r(\Rho)_{li}\big] \Rho_j^r \\
= \left(12 \nabla_r(\Rho)_{il} + 6 \nabla_i(\Rho)_{rl} - 18
\nabla_l(\Rho)_{ri} \right) \Rho_j^r.
\end{multline*}
On the other hand, the first and third term on the right-hand side
of \eqref{8-C} yield
$$
(8 \nabla_r(\Rho)_{li} + 4 \nabla_i(\Rho)_{lr}
- 12 \nabla_l(\Rho)_{ir}) \Rho_j^r \varphi^l \Hess^{ij}(u)
$$
using the symmetry of $\Hess^{ij}(u)$. Now the obvious relation
$$
96 (12,6,-18) = 144 (8,4,-12)
$$
completes the proof. \end{proof}

Next, we describe a second motivation of the $\Omega^{(2)}$-term in
\eqref{P8-general} in general dimensions. For this we recall that
$P_8^0$ is determined by conformal variation of $Q_8$. The
difference
$$
Q_8 - 3! 4! 2^7 v_8
$$
can be expressed in various ways in terms of lower order
constructions. Here one either applies a generalization of the
holographic formula \eqref{hol} or uses a generalization of the
recursive formula \eqref{rec-q-8} to subcritical cases. Now Graham
\cite{G-ext} has shown that\footnote{Here we correct a misprint in
the last term of formula (2.23) in \cite{G-ext}}
\begin{multline}\label{v8}
2^4 v_8 = \tr (\wedge^4 \Rho) \\ + \frac{1}{3} \tr
(\Rho^2\Omega^{(1)}) - \frac{1}{3} \tr (\Rho) \tr (\Rho
\Omega^{(1)}) - \frac{1}{12} \tr (\Rho \Omega^{(2)}) - \frac{1}{12}
\tr (\Omega^{(1)} \Omega^{(1)}).
\end{multline}
In particular, $Q_8$ contains the contribution
$$
-96 (\Omega^{(2)},\Rho).
$$
Now conformal variation yields
$$
96 (\Omega^{(2)},\Hess(\varphi)),
$$
up to a first order operator. This motivates the $\Omega^{(2)}$-term
in \eqref{P8-general}.

A similar argument motivates the $(\Rho \Omega^{(1)} \!+\!
\Omega^{(1)} \Rho)$-term. In fact, by \eqref{v8}, the quantity
$(\Rho^2,\Omega^{(1)})$ contributes to $Q_8$ with the coefficient
$3! 2^6$. Now conformal variation of $(\Rho^2,\Omega^{(1)})$ yields
$-2(\Rho \Omega^{(1)},\Hess(\varphi)) + \cdots$. This motivates the
contribution
$$
3! 2^6 \delta ((\Rho \Omega^{(1)}\!+\!\Omega^{(1)}\Rho) \# d)
$$
in \eqref{P8-general}.

Now universality claims that, in all dimensions $n \ge 8$, the
analogous operator
\begin{equation*}
\P_8^0 - 3! 2^4 \delta(\left[ \Omega^{(2)} - 4 (\Rho \Omega^{(1)} +
\Omega^{(1)} \Rho) + 12 \Rho^3 \right] \# d) + \left(\f\!-\!4\right)
Q_8
\end{equation*}
is conformally covariant, too (and coincides with $P_8$). This
extends Conjecture \ref{M8-gen}. In particular, regarding the
operator as a rational function in $n$, it has a simple pole at
$n=6$. For the residue we find
$$
-96 \delta (\Res_{n=6}(\Omega^{(2)})\# d) + 96
(\Res_{n=6}(\Omega^{(2)}),\Rho) = -48 (\delta (\OB \# d) -
(\OB,\Rho))
$$
by the residue formula (\cite{G-ext}, Proposition 2.8)
$$
2 \Res_{n=6}(\Omega^{(2)}) = \OB.
$$
Here $\OB$ denotes the Fefferman-Graham obstruction tensor in
dimension six. An explicit formula for $\OB$ can be found, for
instance, in \cite{GH}. As the Bach tensor $\B$ in dimension $4$,
$\OB$ is trace-free and divergence-free. Moreover, its
transformation law $e^{4\varphi} \hat {\OB} = \OB$ in dimension six
generalizes $e^{2\varphi} \hat {\B} = \B$ in dimension $4$. A direct
calculation, using these properties, confirms that the operator
$$
\R_8 = \delta(\OB \#d) - (\OB,\Rho) = -(\OB,\Hess) - (\OB,\Rho)
$$
is conformally covariant (in dimension six), i.e.,
$$
e^{7\varphi} \circ \hat{\R}_8 = \R_8 \circ e^{-\varphi}.
$$
It obstructs the existence of $\PP_8$ for general metrics in
dimension six.

Similarly, in dimension $n=4$, the existence of $\PP_8$ is
obstructed by a conformally covariant self-adjoint differential
operator of order four with main part $(\B,\Hess \Delta)$.

Finally, we show that the above results can be seen as special cases
of the following extension of Conjecture \ref{formula}.

\begin{conj}[\bf Universal recursive formulas for GJMS-operators]
\label{Haupt} Let the integer $N \ge 1$ satisfy $2N \le n$ if $n$ is
even. Then on any Riemannian manifold of dimension $n \ge 3$, the
GJMS-operator $P_{2N}$ is given by the recursive formula
\begin{equation}\label{haupt}
P_{2N} = \P_{2N}^0 - a_N \delta( D_{2N-2} \# d) + (-1)^N
\left(\f\!-\!N\right)  Q_{2N}.
\end{equation}
Here $a_N = (2^{N-1}(N-1)!)^2$, and the natural symmetric bilinear
forms $D_2(g), D_4(g), \dots$ are the coefficients of the Taylor
series
$$
g_r^{-1} = g^{-1} + r^2 D_2(g) + r^4 D_4(g) + \cdots
$$
of the inverse of the symmetric bilinear form $g_r$ so that
$r^{-2}(dr^2 + g_r)$ is the Poincar\'e-Einstein metric associated to
$g$.
\end{conj}

We recall that in the locally conformally flat case, $g_r = (1-r^2/2
\Rho)^2$. Hence
$$
g_r^{-1} = \sum_{r \ge 1} N \Rho^{N-1} (r^2/2)^{N-1}.
$$
It follows that
$$
a_N D_{2N-2} = 2^{N-1} N! (N\!-\!1)! \Rho^{N-1}.
$$
This proves that Conjecture \ref{Haupt} extends Conjecture
\ref{formula}.

In order to see that \eqref{haupt} also extends the formulas
\eqref{p6-general} for $P_6$ and \eqref{P8-general} for $P_8$, we
use the fact that for general metrics the coefficients in the Taylor
series
$$
g_r = \sum_{N \ge 0} \left(-\frac{r^2}{2}\right)^N \frac{1}{N!}
g_{(2N)} \st \sum_{N \ge 0} d_{2N} r^{2N}
$$
are given by
\begin{align*}
\frac{1}{2} g_{(2)} & = \Rho, \\
\frac{1}{2} g_{(4)} & = \Omega^{(1)} + \Rho^2, \\
\frac{1}{2} g_{(6)} & = \Omega^{(2)} + 2 (\Rho \Omega^{(1)} +
\Omega^{(1)} \Rho)
\end{align*}
(see \cite{G-ext}, (2.22)). Now we have
$$
g_r^{-1} = g^{-1} + D_2 r^2 + D_4 r^4 + D_6 r^6 + \cdots
$$
with
$$
D_2 = - d_2, \quad D_4 = -d_4 + d_2^2, \quad D_6 = -d_6 + (d_2 d_4 +
d_4 d_2) - d_2^3.
$$
Hence
\begin{align*}
D_2 & = \frac{1}{2} g_{(2)} = \Rho, \\
D_4 & =  -\frac{1}{8} g_{(4)} + \frac{1}{4} g_{(2)}^2 = \frac{1}{4}
(-\Omega^{(1)} + 3 \Rho^2)
\end{align*}
and
\begin{equation*}
D_6 = \frac{1}{48} g_{(6)} - \frac{1}{16} (g_{(2)} g_{(4)} + g_{(4)}
g_{(2)}) + \frac{1}{8} g_{(2)}^3 = \frac{1}{24} (\Omega^{(2)} - 4
(\Omega^{(1)} \Rho + \Rho \Omega^{(1)}) + 12 \Rho^3).
\end{equation*}
This proves the claim.

It is natural to summarize the assertions of Conjecture \ref{Haupt}
for the non-constant terms in form of the identity
\begin{equation}\label{haupt-EF}
\sum_{N \ge 1} \frac{\M_{2N}^0(g)}{(N\!-\!1)!(N\!-\!1)!}
\left(\frac{r^2}{4}\right)^{N-1} = - \delta (g_r^{-1} \# d).
\end{equation}
The latter relation extends \eqref{gf-cf-m}. Of course, in even
dimensions, \eqref{haupt-EF} is to be understood as an identity of
finite series.

By \cite{G-ext}, the symmetric bilinear forms $D_{2N}$ are given by
{\em universal} formula in terms of extended obstruction tensors and
$\Rho$. Combined with Graham's \cite{G-ext} universal formulas for
the holographic coefficients in terms of extended obstruction
tensors and the recursive formulas for $Q$-curvatures (Conjecture
\ref{Q-quadrat}), Conjecture \ref{Haupt} yields universal formulas
for all GJMS-operators in terms of the building blocks $\Rho,
\Omega^{(1)}, \Omega^{(2)},...$.

%%%%%%%%%%%%%%%%%%%%%%%%%%%%%%%%%%%%%%%%%%%%%%%%%%%%%%%%%%%%%%%%%%%%%%%%%
\section{Further comments and open problems}\label{co}

In the locally conformally flat case, Theorem \ref{comm-2} deduces
the conformal covariance of the critical operator
$$
\PP_n = \P_n^0 - c_\f \T_{\f-1} = \Delta^\f + LOT
$$
on a manifold of even dimension $n$ from the relations
\begin{equation}\label{sub-rel}
\M_{2N}^0 = - c_N \T_{N-1}, \; 2N < n
\end{equation}
for all subcritical GJMS-operators. Theorem \ref{M8} provides an
application to a construction of a conformally covariant fourth
power $\PP_8$ of $\Delta$ in dimension $8$. It rests on the fact
that, in this case, the assumptions in Theorem \ref{comm-2}, i.e.,
\eqref{sub-rel} for $N=2$ and $N=3$, are known to be satisfied by
Theorem \ref{M4} and Theorem \ref{M6}.

The identification of $\PP_8$ with $P_8$, however, remains open. As
already mentioned in Section \ref{fourth}, an approach through
conformal variation of $Q_8$ based on \eqref{rec-q-8} seems
feasible. In this connection, the universality of \eqref{rec-q-8}
would play a crucial role.

In an analogous proof of the conformal covariance of
$$
\PP_{10} = \P_{10}^0 - c_5 \T_4
$$
in dimension $n=10$, the only missing piece is a substitute of the
relation \eqref{sub-rel} for $N=4$ in dimension $n=10$. The
principle of universality predicts that this relation actually holds
true in all higher dimensions.

A direct identification of the computer-derived formula \cite{GP}
for $Q_8$ (in general dimension) with the universal recursive
formula \eqref{rec-q-8} is a challenging task. We illustrate the
issue by relating the respective coefficients of $(\Delta^2
(\Rho),\Rho)$ and $\J^4$ in both formulas. In the extension of
\eqref{rec-q-8} to general dimensions, the contribution $(\Delta^2
(\Rho),\Rho)$ appears through
\begin{equation}\label{contri}
-3 P_2(Q_6), \quad (9P_4 - 12P_2^2)(Q_4) \quad \mbox{and} \quad
3!4!2^7 v_8.
\end{equation}
Now the contribution
$$
-2P_2(Q_4) + \frac{16}{n-4}(\B,\Rho)
$$
in $Q_6$ (see \eqref{q6}, \eqref{v6}) yields
$$
-24 \frac{n\!-\!2}{n\!-\!4} (\Delta^2\Rho,\Rho).
$$
The second term in \eqref{contri} gives $12 (\Delta^2 (\Rho),\Rho)$.
Finally, by \eqref{v8} and
$$
(n\!-\!4)(n\!-\!6) \Omega^{(2)} = \Delta^2 (\Rho) + \cdots
$$
(see \cite{G-ext}), the last term in \eqref{contri} gives
$$
- \frac{96}{(n\!-\!4)(n\!-\!6)} (\Delta^2 (\Rho),\Rho).
$$
Adding these results, reproduces the coefficient
$$
-12 \frac{n\!-\!2}{n\!-\!6}
$$
in \cite{GP}. Next, we apply \eqref{rec-q-8} to determine the
coefficient of $\J^4$. A calculation using
$$
Q_2 = \J, \quad Q_4 = \f \J^2 + \cdots \; \mbox{(by \eqref{q4-gen})}
$$
and
$$
Q_6 = \frac{(n\!-\!2)(n\!+\!2)}{4} \J^3 + \cdots \;\mbox{(by
\eqref{q6}, \eqref{v6})}
$$
yields
$$
Q_8 = \frac{(n\!-\!4)n(n\!+\!4)}{8} \J^4 + \cdots.
$$
Note that in this calculation all terms in \eqref{rec-q-8}
contribute in a non-trivial way. These results for $Q_6$ and $Q_8$
fit with \cite{GP}. More generally, the recursive formula predicts
the contribution
$$
\prod_{j=1}^{N-1} \left(\f\!-\!N\!+\!2j\right) \J^N
$$
in $Q_{2N}$, which, in the critical case, is only caused by the
$v_n$-term.

In general, \eqref{Q-gen-rec} identifies $Q_{2N}$ with a sum of the
form
\begin{equation}\label{q-sum}
\Q_{2N} + \cdots + (-1)^N N! (N\!-\!1)! 2^{2N-1} v_{2N}.
\end{equation}
Now infinitesimal conformal variation of this sum yields an operator
of the form $\P_{2N} + \cdots$. Thus, Conjecture \ref{Q-quadrat}
implies a representation formula $P_{2N} = \P_{2N} + \cdots$.
Conjecture \ref{formula} actually predicts huge cancellations in
that sum. It is crucial to understand the mechanism of these
cancellations.

A good understanding of the infinitesimal conformal variation of the
quantity in \eqref{q-sum} would be an important ingredient in a
proof of the conformal covariance of the subcritical operator
\begin{multline}\label{mod-op}
\P_{2N}^0 - c_N \delta(\Rho^{N-1} \# d) \\ + \left(\f\!-\!N\right)
(-1)^N (\Q_{2N} + \cdots + (-1)^N N! (N\!-\!1)! 2^{2N-1} v_{2N}), \;
2N < n
\end{multline}
(in the locally conformally flat case) along similar lines as in the
critical case. The constant term of the operator \eqref{mod-op} is
given by the sum in \eqref{q-sum}. Such a proof is independent of
the recognition of the constant term as $Q_{2N}$. We shall
illustrate this for $N=2$ in Section \ref{app-pan}.

The proofs of the conformal covariance of the respective critical
operators $P_4$, $P_6$ and $\PP_8$ rest only on the conformal
transformation properties of the lower order GJMS-operators in the
respective primary parts, and of the tensors which contribute to the
secondary parts. In particular, the proof of Theorem
\ref{P8-non-flat} does not require explicit formulas for the
extended obstruction tensors. This feature of the proofs nurtures
the hope for a similar treatment of the general case.

The original ambient metric construction \cite{GJMS} generates the
operator $P_{2N}$ from the action of the power
$\Delta_{\tilde{g}}^N$ of the Laplacian of the ambient metric
$\tilde{g}$ on a space of homogenous functions of a certain degree
depending on $N$. Since the functional spaces depend on $N$, even
the very existence of recursive formulas for these operators remains
obscure from this perspective.

The method of recursive constructions of conformally covariant
powers of the Laplacian described here does {\em not} rest on the
ambient metric. Nevertheless, the construction somehow forces the
Taylor coefficients of the ambient metric to appear. In fact,
\eqref{haupt-EF} states that the Taylor coefficients of the
second-order operator on the right-hand side provide appropriate
correction terms which can be used to make the respective operators
$\P_{2N}^0$ conformally covariant. These correction terms contain
the full information on the ambient metric. It seems that there are
no alternative choices for these terms.

It is often fruitful to think of conformally covariant differential
operators on general manifolds as ``curved analogs" of their special
cases on spheres. In the opposite direction, tractor calculus offers
constructions of curved analogs of differential operators on spheres
which are equivariant with respect to the conformal group. A recent
manifestation of this line of thinking is the method of curved
Casimir operators \cite{CGS}. Although the present paper does not
emphasize the perspective of curved analogs, it is tempting to ask
for a representation theoretical interpretation of the operators
$\M_{2N}$. In particular, regarding the identities in Theorem
\ref{sphere} and Theorem \ref{pseudo} as non-linear relations among
intertwining operators for principal series representations (see
\eqref{nonlinear}) of $O(q+1,p+1)$ motivates to ask for
representation theoretical proofs of these relations.

One consequence of the holographic formula \eqref{hol} for the
critical $Q$-curvature $Q_n$ is the proportionality \cite{GZ}
\begin{equation}\label{proportional}
2 \int_{M^n} Q_n vol = 2^n (-1)^\f \left(\f\right)!
\left(\f\!-\!1\right)! \int_{M^n} v_n vol.
\end{equation}
On the other hand, \eqref{Q-gen-rec} predicts that
\begin{equation}\label{crit-Q-con}
2 Q_n - 2^n (-1)^\f \left(\f\right)! \left(\f\!-\!1\right)! v_n  = 2
\Q_n - \sum_{j=1}^{\f-1} \frac{j(\f\!-\!j)}{\f} \binom{\f}{j}^2
\Lambda_{n-2j} \Lambda_{2j},
\end{equation}
where
$$
\Q_n = (-1)^{\f-1} \sum_{|I|+a=\f, \; a \ne \f} (-1)^a m_{(I,a)}
P_{2I}(Q_{2a}).
$$
Hence a proof of the identity
\begin{equation}\label{int-test}
\int_{M^n} \left[ 2 \Q_n - \sum _{j=1}^{\f-1} \frac{j(n\!-\!2j)}{n}
\binom{\f}{j}^2 \Lambda_{n-2j} \Lambda_{2j} \right] vol = 0
\end{equation}
could be regarded as support for \eqref{crit-Q-con}. It would be
interesting to give a {\em direct} proof that the integrand in
\eqref{int-test} is a total divergence.

In small dimensions, this can be verified directly. For instance, a
calculation in dimension $n=6$ shows that the integrand has the form
$$
- 2P_2^0 (Q_4) + 2 P_4^0(Q_2) - 3 P_2^0 P_2(Q_2).
$$
Similarly, we find that the {\em critical} $Q_8$ is given by the sum
of the reduced primary part
\begin{multline*}
[-3 P_2^0(Q_6) - 3 P_6^0(Q_2) + 9 P_4^0(Q_4) \\ + 8P_2^0P_4(Q_2) -
12 P_2^0P_2(Q_4) + 12 P_4^0 P_2 (Q_2) - 18 P_2^0P_2^2(Q_2)],
\end{multline*}
the additional terms
\begin{equation*}
12 [\Delta (Q_4) Q_2 - Q_4 \Delta(Q_2)] + 18 [\Delta^2(Q_2) -
\Delta(Q_2)^2] + 54 [\Delta(Q_2) Q_2^2 - Q_2 \Delta(Q_2^2)]
\end{equation*}
and $3! 4! 2^7 v_8$.

In particular, such a proof requires to verify that, for any
non-trivial composition $I$ of size $\f$, the coefficient of
$Q_{2I}$ in the integrand vanishes. We confirm the vanishing in two
interesting special cases. First, let $I=(p,q)$ with $|I|=\f$. The
coefficient of $Q_{2I}$ in the integrand in \eqref{int-test} is
given by the difference of
$$
- 2 \left( (n/2-p) m_{(p,q)} + (n/2-q) m_{(q,p)} \right)
$$
and
$$
\frac{p(\f\!-\!p)}{\f} \binom{\f}{p}^2 + \frac{q(\f\!-\!q)}{\f}
\binom{\f}{q}^2.
$$
Now the formula
$$
m_{(p,q)} = - \binom{\f-1}{p} \binom{\f-1}{q}
$$
(see \eqref{m-double}) shows that the difference vanishes. Next, let
$I=(i,j,k)$ with $|I|=\f$. \eqref{bino} and Corollary \ref{self}
imply
\begin{equation}\label{rec-3-2}
\frac{m_{(i,j,k)}}{|I|-k} = - \binom{|I|}{k}^2 \frac{k(|I|-k)}{|I|}
\frac{m_{(j,i)}}{|I|-i}.
\end{equation}
The coefficient of $Q_{2I}$ in the integrand in \eqref{int-test} is
the sum of
\begin{equation}\label{first-sum}
-2\left(\f\!-\!i\right)\left(\f\!-\!j\right)\left(\f\!-\!k\right)
\sum_{\sigma \in S_3} \frac{m_{\sigma I}}{\f- (\sigma I)_3}
\end{equation}
and a certain linear combination of
\begin{multline}\label{second-sum}
m_{(j,k)} \left(\f\!-\!j\right) + m_{(k,j)} \left(\f\!-\!k\right), \\
m_{(i,j)} \left(\f\!-\!i\right) + m_{(j,i)} \left(\f\!-\!j\right) \;
\mbox{and} \;  m_{(i,k)} \left(\f\!-\!i\right) + m_{(k,i)}
\left(\f\!-\!k\right).
\end{multline}
Now the relation \eqref{rec-3-2} yields a cancellation of the six
terms in \eqref{first-sum} against the six terms in
\eqref{second-sum}.

We finish with some comments on the generating function $\G$ (see
\eqref{G-def}). The relation
\begin{equation}\label{LL}
\L_{2N} = - (-1)^N \Lambda_{2N}/(N\!-\!1)!, \; N \ge 1
\end{equation}
between $\Lambda_{2N}$ and the leading coefficient of the
polynomials $Q_{2N}^{res}(\lambda)$ (see \eqref{Taylor-q-res})
implies $\G(r) = \L(r)$, where
$$
\L(r) \st - \sum_{N \ge 0} \L_{2N} \frac{r^N}{N!}, \quad \L_0 \st
-1.
$$
In these terms, Conjecture \ref{Q-G} reads
\begin{equation}\label{mystery}
\L(r^2/4) = \sqrt{v(r)}.
\end{equation}
Thus, for the proof of Conjecture \ref{Q-G} it suffices to prove the
relation \eqref{LL} and to establish \eqref{mystery}. \eqref{LL}
should be a consequence of the recursive structure of the
$Q$-polynomials (see the discussion in Section \ref{polynomial}).
For a proof of \eqref{mystery} see \cite{juhl-Q8} (Proposition 4.2).

%%%%%%%%%%%%%%%%%%%%%%%%%%%%%%%%%%%%%%%%%%%%%%%%%%%%%%%%%%%%%%%%%%%%%%%%%%
\section{Appendix}\label{app}

Here we present self-contained proofs of the conformal covariance of
$P_4$ and $P_6$ in the respective critical dimensions and for
general metrics. These proofs serve as illustrations of the general
lines of arguments of this paper.

%%%%%%%%%%%%%%%%%%%%%%%%%%%%%%%%%%%%%%%%%%%%%%%%%%%%%%%%%%%%%%%%%%%%%%%%%%
\subsection{The critical Paneitz operator}\label{app-pan}

In the present section we prove the conformal covariance of the
critical Paneitz operator. We also discuss an extension of the
argument to general dimensions.

Let $n=4$. We write $P_4$ (see \eqref{pan}) in the form
\begin{equation}\label{trick}
P_4 = (P_2^2)^0 - 4 \delta(\Rho \# d) = \P_4^0 - 4 \T_1.
\end{equation}
The following result is a special case of Theorem \ref{c-cv}.

\begin{lemm}\label{cv-4} For $n=4$ and $\P_4 = P_2^2$,
$$
(d/dt)|_0 \left(e^{4t\varphi} \P_4(e^{2t\varphi}g) \right) =
\left[P_2(g),[P_2(g),\varphi]\right].
$$
\end{lemm}

\begin{proof} The relation
$$
e^{3\varphi} P_2(e^{2\varphi}g) = P_2(g) e^{\varphi}
$$
implies
$$
e^{4\varphi} \hat{P}_2^2 = e^{\varphi} P_2 e^{-2\varphi} P_2
e^{\varphi},
$$
where $P_2$ and $\hat{P}_2$ are the respective Yamabe operators for $g$
and $\hat{g} = e^{2\varphi} g$. Hence
\begin{align*}
(d/dt)|_0 (e^{4t\varphi} \P_4(e^{2t\varphi}g)) & = \varphi P_2^2 - 2
P_2 \varphi P_2 + P_2^2 \varphi \\
& = [P_2,[P_2,\varphi]].
\end{align*}
The proof is complete.
\end{proof}

Lemma \ref{cv-4} implies that
$$
(d/dt)|_0 \left(e^{4t\varphi} \P^0_4(e^{2t\varphi}g) \right) =
\left[P_2,[P_2,\varphi]\right]^0 = [\Delta,[\Delta,\varphi]]^0.
$$
But $[\Delta,\varphi] u = 2(d\varphi,du) + u \Delta \varphi$ gives
\begin{align}\label{comm-4}
[\Delta, [\Delta, \varphi]]^0 u & = 2 \Delta (du, d\varphi) + 2 (d
\Delta \varphi, du) - 2 (d\varphi, d \Delta u) \nonumber \\
& = 4 (\Hess(u), \Hess(\varphi)) + 4 (d \Delta\varphi,du) +
4(\Ric,du \otimes d\varphi)
\end{align}
using Weitzenb\"ock's formula. Here $\Hess(X,Y)(u) = \langle
\nabla_X(du),Y\rangle$ is the covariant Hessian of $u$. In order to
determine the conformal variation of $\T_1 = \delta (\Rho \# d)$, we
write $\T_1 (u) = - (d \J,du) - (\Rho, \Hess (u))$. The
transformation laws
$$
e^{2\varphi} \hat{\J} = \J - \Delta \varphi - \left(\f\!-\!1\right)
|d \varphi|^2,
$$
\begin{equation}\label{schouten-c}
\hat{\Rho} = \Rho - \Hess (\varphi) - \frac{1}{2} |d\varphi|^2 g +
d\varphi \otimes d\varphi
\end{equation}
and
\begin{equation}\label{hess-c}
\widehat{\Hess}(u) = \Hess (u) - du \otimes d\varphi - d\varphi
\otimes du + (du,d\varphi) g
\end{equation}
imply
\begin{multline}\label{add-4}
(d/dt)|_0 \left(e^{4t\varphi} \T_1(e^{2t\varphi}g)u\right) =
(d\Delta\varphi,du) + \J(du,d\varphi) \\
+ (\Hess (u),\Hess (\varphi)) + 2 (\Rho,du \otimes \varphi).
\end{multline}

Now combining \eqref{comm-4} and \eqref{add-4} with $\Ric = 2 \Rho +
\J g$, yields
$$
(d/dt)|_0 \left(e^{4t\varphi} P_4(e^{2t\varphi}g)u\right) = 0.
$$
This proves the infinitesimal conformal covariance of $P_4$.

The above calculations also confirm that
\begin{equation}\label{second-4}
4 (d/dt)|_0 \left(e^{4t\varphi} \T_1(e^{2t\varphi}g)\right) =
[\T_0(g),[\T_0(g),\varphi]]^0.
\end{equation}
This is a special case of Theorem \ref{comm-2}.

In the above argumentation, one can replace the calculation of first
order terms by the following reasoning. $P_4$ and $\T_1$ both are
self-adjoint without constant term. It follows that the second-order
differential operators on the right-hand sides of \eqref{comm-4} and
\eqref{add-4} are self-adjoint without constant terms. Two such
operators coincide iff their main parts coincide. Hence it suffices
to compare the coefficients of $\Hess(u)$. In particular, this
argument avoids to invoke Weitzenb\"ock's formula.

Finally, we discuss how the above arguments extend to general
dimensions $n \ge 3$. In that case, the recursive formula for $P_4$
reads
\begin{equation}\label{present-p4}
P_4 = \P_4^0 - 4 \T_1 + \left(\f\!-\!2\right) Q_4.
\end{equation}
This formula follows from \eqref{pan} by direct calculation.

We use \eqref{present-p4} to prove the conformal covariance of
$P_4$. First of all, analogous calculations as above show that
\begin{equation*}
(d/dt)|_0 \left( e^{(\f+2)t\varphi} \P_4(e^{2t\varphi}g)
e^{-(\f-2)t\varphi} \right) = [P_2(g),[P_2(g),\varphi]].
\end{equation*}
This result implies
\begin{multline*}
(d/dt)|_0 \left[ e^{(\f+2)t\varphi} \P_4^0 (e^{2t\varphi}g)
(e^{-(\f-2)t\varphi} u) \right] + \left(\f-2\right) \P_4^0(g)(\varphi) u \\
= [P_2(g),[P_2(g),\varphi]]^0 u.
\end{multline*}
Hence
\begin{multline*}
(d/dt)|_0 \left[ e^{(\f+2)t\varphi} \left(\P_4^0 +
\left(\f\!-\!2\right) Q_4 \right)(e^{2t\varphi}g)
(e^{-(\f-2)t\varphi} u) \right] \\
= - \left(\f-2\right) \P_4^0(g)(\varphi) u + \left(\f-2\right)
(\P_4^0(g) - 4 \T_1(g))(\varphi) u + [P_2(g),[P_2(g),\varphi]]^0 u \\
= [P_2(g),[P_2(g),\varphi]]^0 u - \left(\f\!-\!2\right) 4
\T_1(g)(\varphi) u.
\end{multline*}

Here we have used the variational formula
\begin{equation}\label{var-q4}
(d/dt)|_0 (e^{4t\varphi} Q_4(e^{2t\varphi}g)) = (\P_4^0(g) - 4
\T_1(g))(\varphi)
\end{equation}
for $Q_4$. For the proof of \eqref{var-q4} it is natural to combine
the formula
\begin{equation}\label{present-q4}
Q_4 = -P_2(Q_2) - Q_2^2 + 16 v_4
\end{equation}
with the conformal transformation laws for $P_2$, $Q_2$ and $\Rho$.
The latter formula immediately shows that the left hand side of
\eqref{var-q4} is of the form
$$
P_2^2 - 4 \delta(\Rho \# d) + LOT = \P_4 - 4 \T_1 + LOT.
$$
The actual calculation yields \eqref{var-q4}.

Next, \eqref{second-4} generalizes to
\begin{equation*}
4 (d/dt)|_0 \left( e^{(\f+2)t\varphi} \T_1(e^{2t\varphi}g)
e^{-(\f-2)t\varphi} \right) = [\T_0(g),[\T_0(g),\varphi]]^0 -
\left(\f\!-\!2\right) 4 \T_1(g)(\varphi).
\end{equation*}

Combining these results yields the conformal covariance of $P_4$.

We finish with two comments. In contrast to the critical case, the
proof for $n \ne 4$ requires conformal variation of $Q_4$. Only by
recognizing the right-hand side of \eqref{present-q4} as the fourth
order $Q$-curvature leads to the identification of the right-hand
side of \eqref{present-p4} as the Paneitz operator.

%%%%%%%%%%%%%%%%%%%%%%%%%%%%%%%%%%%%%%%%%%%%%%%%%%%%%%%%%%%%%%%%%%%%%%%%%%%%%%%%%
\subsection{The critical $P_6$ for general metrics}\label{inv-4-6}

Here we prove the infinitesimal conformal covariance of the operator
\begin{equation}\label{op-6}
\PP_6 \st \left[2(P_2 P_4 + P_4 P_2) - 3 P_2^3\right]^0 - 48 \delta
(\Rho^2 \# d) - 8 \delta(\B \#d)
\end{equation}
in dimension $n=6$ (see \eqref{p6-general}). The result implies the
conformal covariance
$$
e^{6\varphi} \PP_6(e^{2\varphi}g) = \PP_6(g).
$$

\begin{lemm}\label{cv-6} Let $n=6$ and $\P_6 = 2(P_2 P_4 + P_4 P_2) - 3 P_2^3$.
Then
\begin{equation}\label{cv-p6}
(d/dt)|_0 \left(e^{6t\varphi} \P_6(e^{2t\varphi}g) \right) = 4
[\M_4(g),[P_2(g),\varphi]] + 2 [P_2(g),[\M_4(g),\varphi]].
\end{equation}
\end{lemm}

\begin{proof} The relations
\begin{equation*}
e^{4\varphi} P_2(e^{2\varphi}g) = P_2(g) e^{2\varphi} \quad
\mbox{and} \quad  e^{5\varphi} P_4(e^{2\varphi}g) = P_4(g)
e^{\varphi}
\end{equation*}
imply
\begin{equation*}
\hat{P}_2 \hat{P}_4 = e^{-4\varphi} P_2 e^{-3\varphi} P_4 e^\varphi,
\quad \hat{P}_4 \hat{P}_2 = e^{-5\varphi} P_4 e^{-3\varphi} P_2
e^{2\varphi}
\end{equation*}
and
$$
\hat{P}_2^3 = e^{-4\varphi} P_2 e^{-2\varphi} P_2 e^{-2\varphi} P_2
e^{2\varphi}.
$$
Hence
\begin{align*}
(d/dt)|_0 \left(\P_6(e^{2t\varphi}g) \right) & = 2 \left( -4
\varphi P_2 P_4 - 3 P_2 \varphi P_4 + P_2 P_4 \varphi \right) \\
& + 2 \left (-5 \varphi P_4 P_2 - 3 P_4 \varphi P_2 + 2 P_4 P_2
\varphi \right) \\
& - 3 \left( -4\varphi P_2^3 - 2 P_2 \varphi P_2^2 - 2 P_2^2 \varphi
P_2 + 2 P_2^3 \varphi \right).
\end{align*}
The latter formula is equivalent to
\begin{align*}
(d/dt)|_0 \left(e^{6t\varphi} \P_6(e^{2t\varphi}g) \right) & = 2
\left( -2 [P_2,\varphi] P_4 + P_2 [P_4,\varphi]\right) \\
& + 2 \left( -[P_4,\varphi]P_2 + 2 P_4 [P_2,\varphi]\right) \\
& - 3 \left( -2[P_2,\varphi]P_2^2 + 2 P_2^2 [P_2,\varphi]\right) \\
& = 4 \left[P_4,[P_2,\varphi]\right] + 2
\left[P_2,[P_4,\varphi]\right] - 6 \left[P_2^2,[P_2,\varphi]\right].
\end{align*}
Now $\left[P_2^2,[P_2,\varphi]\right] =
\left[P_2,[P_2^2,\varphi]\right]$ yields the assertion.
\end{proof}

A similar calculation shows that in general dimensions,
\begin{equation*}
(d/dt)|_0 \left(e^{(\f+3)t\varphi} \P_6(e^{2t\varphi}g)
e^{-(\f-3)t\varphi} \right) = 4 [\M_4,[P_2,\varphi]] + 2
[P_2,[\M_4,\varphi]].
\end{equation*}

Lemma \ref{cv-6} implies that
\begin{equation}\label{6-short}
(d/dt)|_0 \left(e^{6t\varphi} \P_6^0(e^{2t\varphi}g) \right) = 4
[\M_4,[P_2,\varphi]]^0 + 2 [P_2,[\M_4,\varphi]]^0.
\end{equation}
The right-hand side coincides with
$$
4 [\M^0_4,[\M^0_2,\varphi]]^0 + 2 [\M^0_2,[\M^0_4,\varphi]]^0.
$$
Using
$$
\M_4^0 = - 4\delta(\Rho \# d) = - 4 \T_1 \quad \mbox{and} \quad
\M_2^0 = \Delta
$$
this sum equals
\begin{equation}\label{sum-6}
-16 [\T_1,[\Delta,\varphi]]^0 - 8 [\Delta,[\T_1,\varphi]]^0.
\end{equation}
As a self-adjoint second-order operator which annihilates constants,
this operator is determined by its main part. A calculation shows
that
$$
[\T_1,[\Delta,\varphi]] u = - 2 (\Rho,\Hess (d\varphi,du)) + 2
(d\varphi, d(\Rho,\Hess (u))) + \mbox{first order terms}.
$$
The main part of this operator is given by
$$
-4\Rho_{ij} \Hess^i_k(\varphi) \Hess^{jk}(u) + 2 \nabla_i(\Rho)_{jk}
\Hess^{jk}(u) \varphi^i.
$$
Similarly, for the second term in \eqref{sum-6} we find
$$
[\Delta,[\T_1,\varphi]] u = -2 \Delta (\Rho, d\varphi \otimes du) +
2 (\Rho,d\varphi \otimes d\Delta u) + \mbox{first order terms}.
$$
The main part of this operator is given by
$$
-4 \nabla_k(\Rho)_{ij} \varphi^i \Hess^{jk}(u) - 4 \Rho_{ij}
\Hess^{ik}(\varphi) \Hess^j_k(u).
$$
Thus, for the main part of \eqref{sum-6} we find the formula
\begin{multline}\label{A}
96 \Rho_{ij} \Hess^{ik}(\varphi) \Hess^j_k(u) - 32
\nabla_i(\Rho)_{jk} \Hess^{jk}(u) \varphi^i + 32
\nabla_k(\Rho)_{ij} \Hess^{jk}(u) \varphi^i \\
= 96 \Rho_{ij} \Hess^{ik}(\varphi) \Hess^j_k(u) + 32 \Y_{kij}
\Hess^{jk}(u) \varphi^i.
\end{multline}
Next, the main part of $(d/dt)|_0 \left(e^{6t\varphi}
\T_2(e^{2t\varphi}g) \right)$ is
\begin{equation}\label{B}
\Hess_{ik}(\varphi) \Rho_j^k \Hess^{ij}(u) + \Rho_{ik}
\Hess_j^k(\varphi) \Hess^{ij}(u) = 2 \Hess_{ik}(\varphi) \Rho_j^k
\Hess^{ij}(u).
\end{equation}
Finally, by the transformation law
\begin{equation}\label{bach-conform}
e^{2\varphi} \hat{\B}_{ij} = \B_{ij} - (n\!-\!4) (\Y_{ikj} +
\Y_{jki}) \varphi^k + (n\!-\!4) \C_{kijl}\varphi^k \varphi^l
\end{equation}
for the Bach tensor, the main part of the conformal variation of
$\delta(\B\#du)$ is
\begin{equation}\label{C}
4 \Y_{ikj} \varphi^k \Hess^{ij}(u).
\end{equation}
Now
$$
\eqref{A} - 48 \times \eqref{B} - 8 \times \eqref{C} = 0
$$
implies the infinitesimal conformal covariance of the operator
\eqref{op-6}. \hfill $\square$
\medskip

The proof shows that the only property of the Bach tensor $\B_{ij}$
which enters is that its conformal variation is given by a multiple
of
$$
(\Y_{ikj} + \Y_{jki}) \varphi^k.
$$
In the proof of Theorem \ref{P8-non-flat}, a similar property of
$\Omega^{(2)}$ plays an analogous role (with the symmetrized Cotton
tensor replaced by the higher Cotton tensor $\Y^{(2)}$).

The above calculations directly confirm the following special case
of Theorem \ref{comm-2}. We recall that $\T_0 = -\Delta$.

\begin{corr}\label{t-6} In the locally conformally flat case,
$$
6 (d/dt)|_0 \left(e^{6t\varphi} \T_2(e^{2t\varphi}g) \right) =
[\T_0,[\T_1,\varphi]]^0 + 2 [\T_1,[\T_0,\varphi]]^0.
$$
\end{corr}

\begin{proof} This is an identity of self-adjoint second-order
operators which annihilate constants. It suffices to compare the
main parts of both sides. The result follows from the above
calculations.
\end{proof}

For any constant $\alpha$, the operator $P_4' = P_4 + \alpha |\C|^2$
is conformally covariant. The same proof as above shows that in
dimension $n=6$,
$$
P_6' = (2(P_2 P_4' + P_4' P_2) - 3 P_2^3)^0 - 48 \delta (\Rho^2 \#
d) - 8 \delta(\B \#d)
$$
is conformally covariant. The latter conformally covariant cube of
the Laplacian differs from the GJMS-operator $P_6$ by a multiple of
the self-adjoint operator $\delta (|\C|^2 du)$. But note that the
latter operator is not conformally covariant in dimensions $n \ne
6$. In other words, universality is lost.

\end{document}